%% file: rob_phase_retrieval.tex
\documentclass[12pt]{article}
\usepackage[margin =1in]{geometry}
\usepackage{amssymb,amsthm}
\usepackage{amsmath}
\usepackage{enumerate}
\usepackage{hyperref}
\usepackage{cite}

\numberwithin{equation}{section}

\newcommand{\ip}[1] {\langle #1 \rangle }

\newcommand{\inclu}[0] {\ar@{^{(}->}}

\newcommand{\spann}{\text{span}}

\newcommand{\gph}{{\rm gph}\,}
\newcommand{\diag}{{\rm diag}}
\newcommand{\dist}{{\rm dist}}
\newcommand{\R}{\mathbb{R}}

\newcommand{\EE}{\mathbb{E}}
\newcommand{\trace}{\mathrm{Tr}}
\newcommand{\sign}{\mathrm{sign}}
\newcommand{\lip}{\mathrm{lip}}

\newcommand{\RR}{\mathbb{R}}
\newcommand{\pfail}{p_{\mathrm{fail}}}

\newcommand{\argmin}{\operatornamewithlimits{argmin}}
\newcommand{\ls}{\operatornamewithlimits{limsup}}
\newcommand{\Diag}{{\rm Diag}}

\newtheorem{thm}{Theorem}[section]
\newtheorem{definition}[thm]{Definition}
\newtheorem{proposition}[thm]{Proposition}
\newtheorem{lem}[thm]{Lemma}
\newtheorem{cor}[thm]{Corollary}

\newtheorem{assumption}{Assumption}

\theoremstyle{remark}
\newtheorem{claim}{Claim}

\usepackage{mathtools}
\DeclarePairedDelimiter{\dotp}{\langle}{\rangle}
\usepackage[boxruled]{algorithm2e}

\begin{document}
	
	\title{The nonsmooth landscape of  phase retrieval}
	
	
	\author{Damek Davis\thanks{School of Operations Research and Information Engineering, Cornell University,
Ithaca, NY 14850, USA;
\texttt{people.orie.cornell.edu/dsd95/}.}
		\and 
		Dmitriy Drusvyatskiy\thanks{Department of Mathematics, U. Washington, 
		Seattle, WA 98195; \texttt{www.math.washington.edu/{\raise.17ex\hbox{$\scriptstyle\sim$}}ddrusv}. Research of Drusvyatskiy was supported by the AFOSR YIP award FA9550-15-1-0237 and by the NSF DMS   1651851 and CCF 1740551 awards.}
		\and Courtney Paquette\thanks{Industrial and Systems Engineering Department, Lehigh University, Bethlehem, PA 18015; \texttt{sites.math.washington.edu/{\raise.17ex\hbox{$\scriptstyle\sim$}}yumiko88/}.}
	}

	\date{}
	\maketitle
	\input{abstract}

\input{intro}

		\input{algorithm}

		\input{numerics}

\input{landscape}

\input{concentration}

\paragraph{Acknowledgments} We thank Peng Zheng (University of Washington, Seattle) for invaluable help in developing the numerical illustrations.

	\bibliographystyle{plain}
	\bibliography{bibliography}

\appendix	

\section*{\huge Appendices}

\section{Auxiliary computations}\label{sec:appendix_aux}

\begin{proof}[Proof of Lemma~\ref{lem:cray_compute}] We let $\sigma_1
  = y_1$ and $\sigma_2 = -y_2$. We may write
\begin{align*}
\EE_v \left[| \sigma_1 v_1^2 - \sigma_2 v_2^2|\right] &= \frac{1}{2\pi}
  \int_{\R^2} |\sigma_1 v_1^2 -\sigma_2 v_2^2 | \, 
  \text{exp} \left (- \left ( \frac{v_1^2 + v_2^2}{2} \right ) \right
  ) \, dv_1 dv_2\\
& = \frac{1}{2\pi} \int_{R_1} \left (\sigma_1 v_1^2 - \sigma_2 v_2^2 \right ) \text{exp} \left (- \left ( \frac{v_1^2 + v_2^2}{2} \right ) \right
  ) \, dv_1 dv_2\\
& \qquad + \frac{1}{2 \pi} \int_{R_2} \left ( \sigma_2 v_2^2 - \sigma_1 v_1^2 \right ) \text{exp} \left (- \left ( \frac{v_1^2 + v_2^2}{2} \right ) \right
  ) \, dv_1 dv_2
\end{align*}
where 
\begin{align*}
R_1 &= \left \{ (v_1,v_2) \, : \, \sqrt{\sigma_1} |v_1| \ge
      \sqrt{\sigma_2} |v_2| \right \}\\
R_2 &= \left \{ (v_1, v_2)\, :\, \sqrt{\sigma_2} |v_2| \ge
      \sqrt{\sigma_1} |v_1| \right \}.
\end{align*}
Using the convention $\arctan(\theta) \in \left [ \frac{-\pi}{2},
  \frac{\pi}{2} \right ]$, we define the angle $\theta_1 := \arctan
\left ( \sqrt{\frac{\sigma_1}{\sigma_2}} \right )$. Passing to the polar coordinates, we deduce 
\begin{align*}
\frac{1}{2 \pi} \int_{R_1} (\sigma_1 v_1^2 - \sigma_2 v_2^2) \,
  \text{exp} \left (- \left ( \frac{v_1^2+v_2^2}{2} \right ) \right )&
  \, dv_1 dv_2\\
&= \frac{1}{2 \pi} \int_{R_1} r^3 ( \sigma_1
  \cos^2(\theta) - \sigma_2 \sin^2(\theta) ) \, e^{-r^2/2} \, dr d\theta.
\end{align*}
We break up the region $R_1$ into three wedges corresponding to the angles $[0, \theta_1]$, $[2\pi, 2\pi-\theta_1]$, and $[\pi + \theta_1,
\pi - \theta_1]$. We will compute the integral over one of the regions. The rest will follow analogously. To this end, we successively deduce
\begin{align*}
\frac{1}{2 \pi} \int_0^{\theta_1} \int_0^\infty r^3 \bigg ( \sigma_1
  \cos^2(\theta) - \sigma_2 \sin^2(\theta) \bigg ) &e^{-r^2/2} \, dr
  d\theta\\
&= \frac{1}{2\pi} \int_0^{\theta_1} \sigma_1 \left (
  1+ \cos(2 \theta) \right ) - \sigma_2 \left (1- \cos(2 \theta)
  \right ) d \theta\\
&= \frac{1}{2 \pi} \big ( (\sigma_1-\sigma_2) \theta + (\sigma_1 +
  \sigma_2) \sin(\theta) \cos(\theta) \big ) \bigg |_0^{\theta_1}
  \\
&= \frac{1}{2 \pi} \big ( (\sigma_1-\sigma_2) \theta_1 + (\sigma_1
  + \sigma_2) \sin(\theta_1) \cos(\theta_1) \big )\\
\frac{1}{2 \pi} \int_{2\pi- \theta_1}^{2\pi} \int_0^\infty r^3 \bigg ( \sigma_1
  \cos^2(\theta) - \sigma_2 \sin^2(\theta) \bigg )& e^{-r^2/2} \, dr
  d\theta\\
&= \frac{1}{2 \pi} \big ( (\sigma_1-\sigma_2) \theta_1 + (\sigma_1
  + \sigma_2) \sin(\theta_1) \cos(\theta_1) \big )\\
\frac{1}{2 \pi} \int_{\pi- \theta_1}^{\pi+ \theta_1} \int_0^\infty r^3 \bigg ( \sigma_1
  \cos^2(\theta) - \sigma_2 \sin^2(\theta) \bigg ) &e^{-r^2/2} \, dr
  d\theta\\
&= \frac{1}{2 \pi} \big ( 2(\sigma_1-\sigma_2) \theta_1 + 2(\sigma_1
  + \sigma_2) \sin(\theta_1) \cos(\theta_1) \big ).
\end{align*}
Similarly, we see that for the region $R_2$, we have 
\begin{align*}
\frac{1}{2 \pi} \int_{R_2} (\sigma_2 v_2^2 - \sigma_1 v_1^2) \,
  \text{exp} \left (- \left ( \frac{v_1^2+v_2^2}{2} \right ) \right )
& \, dv_1 dv_2\\
&= \frac{1}{2 \pi} \int_{R_2} r^3 ( \sigma_2
  \sin^2(\theta) - \sigma_1 \cos^2(\theta) ) \, e^{-r^2/2} \, dr d\theta.
\end{align*}
We break up the region $R_2$ into two wedges where the angles range
from $[\theta_1, \pi-\theta_1]$ and $[\pi + \theta_1, 2\pi -
\theta_1]$ as we did in $R_1$. We will show the explicit computation
for one of these terms and note the rest following using similar
computations:
\begin{align*}
\frac{1}{2 \pi} \int_{\theta_1}^{\pi-\theta_1} \int_0^\infty r^3 \bigg ( \sigma_2
  \sin^2(\theta) - &\sigma_1 \cos^2(\theta) \bigg ) e^{-r^2/2} \, dr
  d\theta\\
&= \frac{1}{2\pi} \int_{\theta_1}^{\pi-\theta_1} \sigma_2 \left (
  1- \cos(2 \theta) \right ) - \sigma_1 \left (1+ \cos(2 \theta)
  \right ) d \theta\\
&= \frac{1}{2 \pi} \big ( (\sigma_2-\sigma_1) \theta -(\sigma_1 +
  \sigma_2) \sin(\theta) \cos(\theta) \big ) \bigg |_{\theta_1}^{\pi-\theta_1}
  \\
&= \frac{1}{2 \pi} \big ( (\sigma_2-\sigma_1) (\pi -2\theta_1) + 2(\sigma_1
  + \sigma_2) \sin(\theta_1) \cos(\theta_1) \big )\\
\frac{1}{2 \pi} \int_{\pi+\theta_1}^{2\pi-\theta_1} \int_0^\infty r^3 \bigg ( \sigma_2
  \sin^2(\theta) - &\sigma_1 \cos^2(\theta) \bigg ) e^{-r^2/2} \, dr
  d\theta\\
& = \frac{1}{2 \pi} \big ( (\sigma_2-\sigma_1) (\pi -2\theta_1) + 2(\sigma_1
  + \sigma_2) \sin(\theta_1) \cos(\theta_1) \big ).
\end{align*}
By combining the computed integrals, we arrive at the full answer
\begin{align*}
\EE_v \left[| \sigma_1 v_1^2 - \sigma_2 v_2^2|\right] &= \frac{4}{\pi}
                                                        \left [
                                                        (\sigma_1-\sigma_2)
                                                        \arctan \left
                                                        (
                                                        \sqrt{\frac{\sigma_1}{\sigma_2}
                                                        }\right ) +
                                                        \sqrt{\sigma_1
                                                        \sigma_2}
                                                        \right ] -
                                                        (\sigma_1-\sigma_2),
\end{align*}
as claimed.
\end{proof}

\begin{proof}[Proof showing Equation~\eqref{eq:large_triangle} implies
  Corollary~\ref{cor: subsample_stuff}] We observe that $D_x \le 2
  \|x\| + 2 \|\bar x\|$, which by
  Claim~\ref{claim:upper-bound_on-size} gives $D_x \le (2\tau + 2)
  \|\bar x\|$. First by applying the triangle inequality with $\|\hat
  x - x\| \le C' \sqrt{\Delta} D_x$ and \eqref{eq:large_triangle}, we obtain 
\begin{align*}
| \|x\| - c\|\bar x\| | \le | \|x-\hat x\| + \|\hat x\| - c \|\bar x\|
  |\lesssim C' \sqrt{\Delta} D_x + \sqrt{\Delta} D_x \frac{\|\bar x\|}{\|
  \hat x\|}.
\end{align*}
Using the bound on $D_x$ gives the desired inequality. Next, we
conclude 
\begin{align*}
|\dotp{x, \bar x}| &\le \|\bar x\| \|x - \hat x\| + |\dotp{\hat x, \bar
  x}| \\
&\lesssim C' \sqrt{\Delta} D_x \|\bar x\| + \sqrt{\Delta} D_x \|\bar x\|.
\end{align*}
Applying the bound on $D_x$, the result is shown. Lastly, using $\|x\|
\le \tau \|\bar x\|$ and $\| \hat x \| \lesssim \|\bar x\|$, we conclude
\begin{align*}
\|x\| \|x-\bar x\| \|x + \bar x\| &\le ( \|x - \hat x\| + \|\hat x\|)
  ( \|x-\bar x\| \|x+\bar x\| )\\
&\le D_x^2 \|x-\hat x\| + \|\hat x\|\|x-\bar x\| \|x+\bar x\|\\
&\lesssim \|\bar x\|^2 D_x \sqrt{\Delta} + \|\hat x\| ( \|\hat x-\bar
  x\| + \|\hat x-x\|  ) \|x+\bar x\|\\
&\lesssim \|\bar x\|^3 \sqrt{\Delta} + \|\bar x\|^2 \sqrt{\Delta} D_x + \|\hat x\|\|\hat x-\bar
  x\| \|x+\bar x\|\\
&\lesssim \|\bar x\|^3 \sqrt{\Delta} + \|\hat x\|\|\hat x-\bar
  x\| (\|x-\hat x\| + \| \hat x +\bar x\|)\\
&\lesssim \|\bar x\|^3 \sqrt{\Delta} + \|\bar x\|^2 D_x \sqrt{\Delta}
  + \|\hat x\|\|\hat x-\bar
  x\| \| \hat x +\bar x\|\\
&\lesssim \|\bar x\|^3 \sqrt{\Delta} + \sqrt{\Delta} D_x \|\bar x\|^2.
\end{align*}
Dividing through by $\|\bar x\|^3$, finishes the proof.
\end{proof}

\input{cray_computations}

\end{document}

%% file: abstract.tex
\begin{abstract}
We consider a popular nonsmooth formulation of the real phase retrieval problem. We show that under standard statistical assumptions, a simple subgradient method converges linearly when initialized within a constant relative distance of an optimal solution. Seeking to understand the distribution of the stationary points of the problem, we complete the paper by proving that as the number of Gaussian measurements increases, the stationary points converge to a codimension two set, at a controlled rate.  Experiments on image recovery problems illustrate the developed algorithm and theory.
\end{abstract}

\bigskip

{\bf Keywords:} Phase retrieval, stationary points, subdifferential, variational principle, subgradient method, spectral functions, eigenvalues

%% file: intro.tex
	\section{Introduction}		
	Phase retrieval is a common task in computational science, with numerous applications including imaging, X-ray crystallography, and speech processing. In this work, we consider a popular real counterpart of the problem. Given a set of tuples $\{(a_i,b_i)\}_{i=1}^m\subset\R^d\times \R$, the (real) phase retrieval problem seeks to determine a vector $x\in \R^d$ satisfying $(a^Tx)^2=b_i$ for each index $i=1,\ldots,m$. Due to its combinatorial nature, this problem is known to be NP-hard \cite{phae_NPhard}. One can model the real phase retrieval problem in a variety of ways.
Here, we consider the following ``robust formulation'':
		$$\min_x~ f_{S}(x):=\frac{1}{m}\sum_{i=1}^m |(a_i^Tx)^2-b_i|.$$

This  model of the problem has gained some attention recently with the work of Duchi-Ruan \cite{duchi_ruan_PR} and Eldar-Mendelson \cite{eM}. Indeed, this model exhibits a number of desirable properties, making it amenable to numerical methods. Namely, in contrast to other possible formulations, mild statistical assumptions imply that $f_S$ is both {\em weakly convex} \cite[Corollary 3.2]{duchi_ruan_PR} and {\em sharp} \cite[Theorem 2.4]{eM}, with high probability. That is, there exist numerical constants $\rho,\kappa>0$ such that 	
$$\textrm{the assignment }x\mapsto f_S(x)+\frac{\eta}{2}\|x\|^2 \textrm{ is a convex function},$$ and the inequality 
$$f_S(x)-\inf f_S\geq \kappa\|x-\bar x\|\|x+\bar x\|\qquad \textrm{ holds for all }x\in \R^d.$$
Here, $\pm \bar x$ are the true signals and $\|\cdot\|$ denotes the $\ell_2$-norm. Weak convexity is a well studied concept in optimization literature \cite{prox_reg,federer,clarkelower,favorablelowerc2}, while sharpness and the closely related notion of error bounds \cite{weak_sharp,prox_error,Luo1993} classically underly rapid local convergence guarantees in nonlinear programming.
 Building on these observations, Duchi and Ruan~\cite{duchi_ruan_PR} showed that with proper initialization, the so-called  \emph{prox-linear algorithm} \cite{comp_DP,prox_error,prox,duchi_ruan,duchi_ruan_PR}  quadratically converges to $\pm \bar x$ (even in presence of outliers). The only limitation of their approach is that the prox-linear method requires, at every iteration, invoking an iterative solver for a convex subproblem. For large-scale instances $(m \gg 1, d \gg 1$), the numerical resolution of such problems is non-trivial. In the current work, we analyze a lower-cost alternative when there are no errors in the measurements.


We will show that the robust phase retrieval objective favorably lends itself to classical subgradient methods. This is somewhat surprising because, until recently, convergence rates of subgradient methods in nonsmooth, nonconvex optimization have remained elusive; see the discussion in \cite{prixm_guide_subgrad}. We will prove that under mild statistical assumptions and proper initialization, the standard Polyak subgradient method 
$$x_{k+1}=x_k-\left(\frac{f_S(x_k)-\min f_S}{\|g_k\|^2}\right)g_k\qquad \textrm{with}\qquad g_k\in \partial f_S(x_k),$$
linearly converges to $\pm \bar x$, with high probability. We note that 
high quality initialization, in turn, is straightforward to obtain; see e.g. \cite[Section 3.3]{duchi_ruan_PR} and \cite{eldar_init}.
The argument we present is appealingly simple, relying only on weak convexity and sharpness of the function.

Aside from the current work and that of \cite{duchi_ruan_PR}, we are not aware of  other attempts to optimize the robust phase retrieval objective directly. Other works focus on different problem formulations. Notably,  Cand\`{e}s et al. \cite{wirt_flow} and Sun et al. \cite{phase_nonconv} optimize the smooth loss $\frac{1}{m}\sum_{i=1}^m (\dotp{a_i, x}^2 - b_i)^2$ using a second-order trust region method and a gradient method, respectively. Wang et al. \cite{eldar_init} instead minimize the highly nonsmooth function $\frac{1}{m}\sum_{i=1}^m (|\dotp{a_i, x}| -\sqrt{b_i})^2$ by a gradient descent-like method. Another closely related recent work is that of Tan and Vershynin~\cite{versh_kac}. One can interpret their scheme as  a stochastic subgradient method on the formulation $\frac{1}{m}\sum_{i=1}^m ||\dotp{a_i, x}|- \sqrt{b_i}|$, though this is not explicitly stated in the paper. Under proper initialization and assuming that $a_i$ are uniformly sampled from a sphere, they prove linear convergence. Their argument relies on sophisticated probabilistic tools. In contrast, we disentangle the probabilistic statements (weak convexity and sharpness) from the deterministic convergence of Algorithm~\ref{alg:polyak}. As a proof of concept, we illustrate the proposed subgradient method synthetic and large-scale real image recovery problems.




Weak convexity and sharpness, taken together, imply existence of a small neighborhood $\mathcal{X}$ of $\{\pm\bar x\}$ devoid of extraneous stationary points of $f_S$ (see Lemma~\ref{lem:region_without_stat}). On the other hand, it is intriguing to determine where the objective function $f_S$ may have stationary points outside of this neighborhood. We complete the paper by proving that as the number of Gaussian measurements increases, the stationary points of the problem converge to a codimension two set, at a controlled rate. This suggests that there are much larger regions than the  neighborhood $\mathcal{X}$, where the objective function has benign geometry. 

 We follow an intuitive and transparent strategy. Setting the groundwork, assume that $a_i$ are i.i.d samples from a normal distribution $\mathsf{N}(0,I_{d\times d})$. Hence the problem $\min f_S$ is an empirical average approximation of the population objective
$$\min_x~f_P(x):=\mathbb{E}_{a}[|(a^Tx)^2-(a^T\bar x)^2|].$$
 Seeking to determine the location of stationary points of $f_S$, we begin by first determining the stationary points of $f_P$. We base our analysis on the elementary observation that $f_P(x)$ depends on $x$ only through the eigenvalues of the rank two matrix $X:=xx^T-\bar x\bar x^T$. More precisely, equality holds:	
 $$f_{P}(x)=\frac{4}{\pi}\left [
\trace(X)\cdot
\arctan \left
(
\sqrt{\left|\frac{\lambda_{\max}(X)}{\lambda_{\min}(X)}\right|
}\right ) +
\sqrt{|\lambda_{\max}(X)
	\lambda_{\min}(X)|}
\right ] -
\trace(X).$$
See Figure~\ref{fig:pop_obj} for a graphical illustration.

\begin{figure}[!tbp]
	\centering
	\begin{minipage}[b]{0.59\textwidth}
		\includegraphics[width=\textwidth]{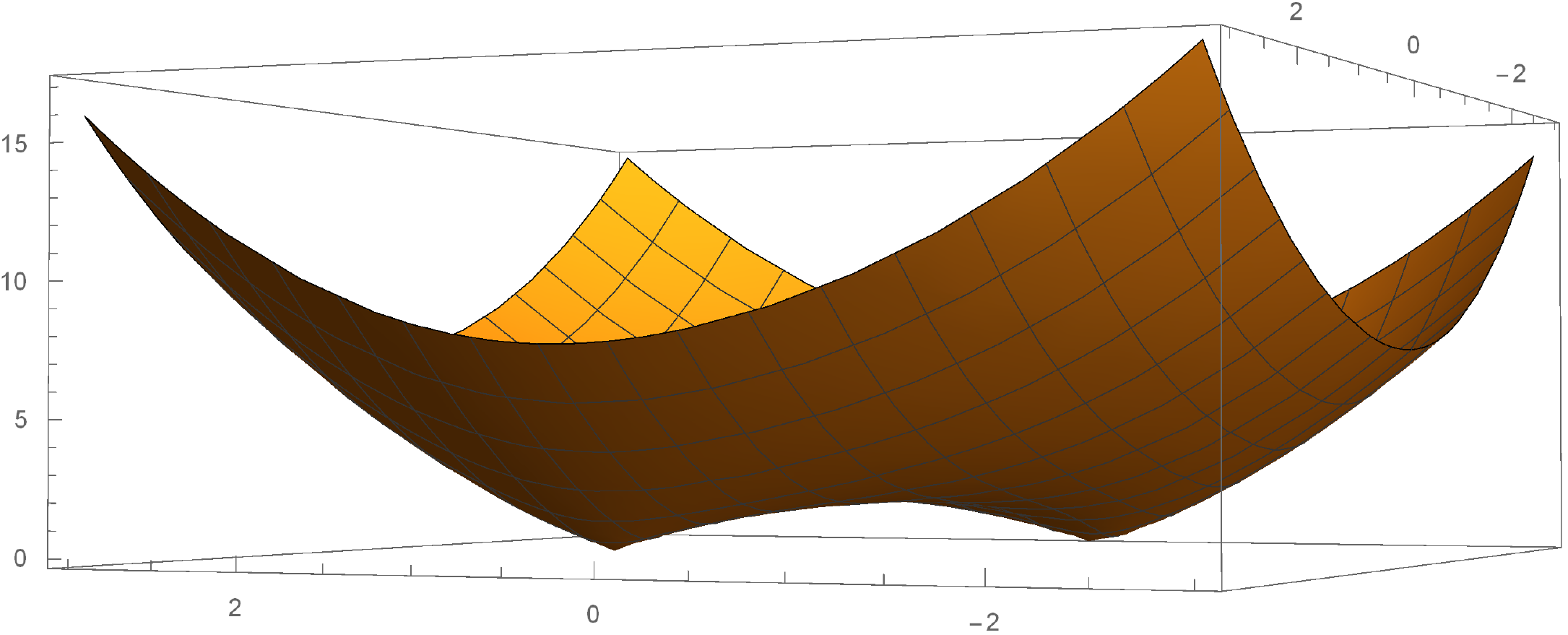}
	\end{minipage}~
	\begin{minipage}[b]{0.4\textwidth}
		\includegraphics[width=\textwidth]{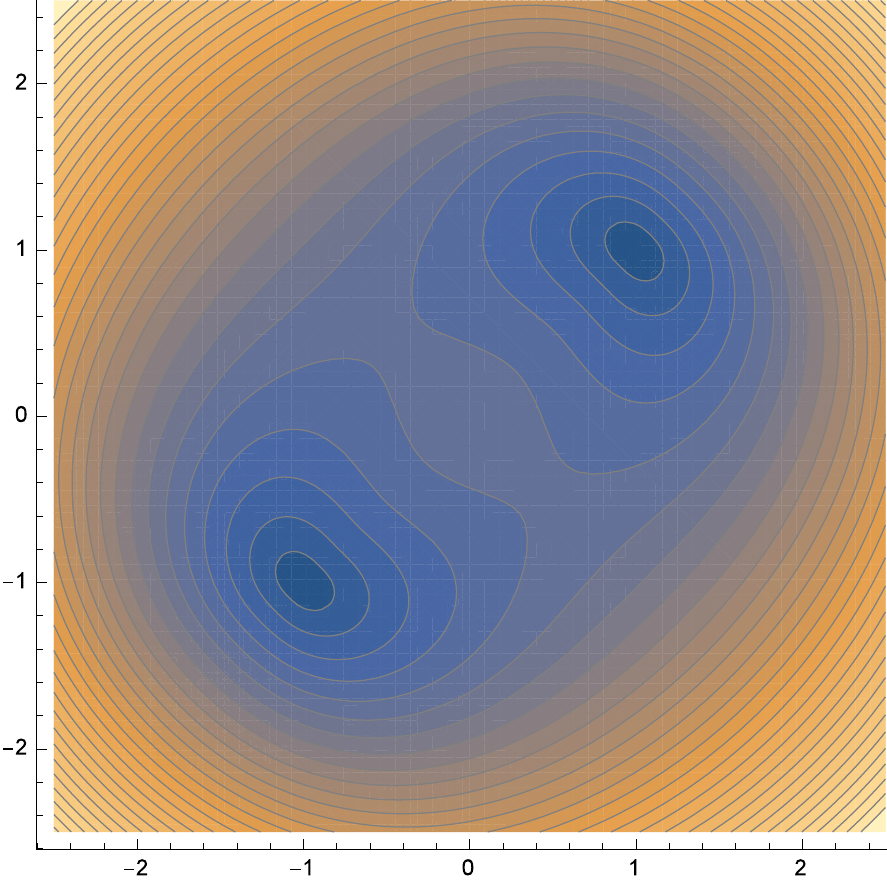}
	\end{minipage}
		\caption{Depiction of the population objective $f_P$ with $\bar x=(1,1)$: graph  (left), contours  (right).}
		\label{fig:pop_obj}
\end{figure}

Using basic perturbation properties of eigenvalues, we will show that the stationary points of $f_P$ are precisely \begin{equation}\label{eq:stationary_pop}\{0\}\cup\{\pm\bar x\}\cup \{x\in \bar x^{\perp}: \|x\|=c\cdot\|\bar x\|\},\end{equation}
where $c \approx 0.4416$ is a numerical constant.  Intuitively, this region, excluding $\{\pm \bar x\}$, is where numerical methods may stagnate. In particular, $f_P$ has no extraneous stationary points outside of the subspace ${\bar x}^{\perp}$.
Along the way, we prove a number of results in matrix theory, which may be of independent interest. For example, we show that all stationary points of a composition of an orthogonally invariant gauge function with the map $x\mapsto xx^T-\bar x\bar x^T$ must be either perpendicular or collinear with $\bar x$. 

Having located the stationary points of the population objective $f_P$,  we turn to the stationary points of the subsampled function $f_S$.  This is where the techniques commonly used for smooth formulations of the problem, such as those in \cite{phase_nonconv}, are no longer applicable; indeed, the subdifferential $\partial f_P(x)$ is usually a very poor approximation of $\partial f_S(x)$. Nonetheless, we show that the {\em graphs} of the subdifferentials $\partial f_P$ and $\partial f_S$ are close with high probability -- a result closely related to the celebrated Attouch's convergence theorem \cite{attouch}. The analysis of the stationary points of the subsampled objective flows from there. Namely, we show that there is a  constant $C$ such that whenever $m \geq Cd$, all stationary points $x$ of $f_S$ satisfy
\begin{equation*}
\frac{\| x\|\| x-\bar x\|\| x+\bar x\|}{\|\bar x\|^3}\lesssim \sqrt[4]{\frac{d}{m}}
\qquad \textrm{or}\qquad \left\{\begin{aligned}
\left|\tfrac{\| x\|}{\|\bar x\|}-c \right|&\lesssim \sqrt[4]{\tfrac{d}{m}}\cdot \left( 1+\tfrac{\|\bar x\|}{\| x\|}\right)\\
\tfrac{|\langle  x,\bar x\rangle|}{\|x\| \|\bar x\|}&\lesssim \sqrt[4]{\tfrac{d}{m}} \cdot\tfrac{\|\bar x\|}{\|x\|}
\end{aligned}\right\},
\end{equation*}	
with high probability; compare with~\eqref{eq:stationary_pop}. The argument we present is very general, relying only on weak convexity and concentration of  $f_S$ around its mean. Therefore, we believe that the technique  may be of independent interest. 

We comment in Section~\ref{sec:comments_robust}  on the structure of stationary points for the variant of the phase retrieval problem, in which the measurements $b$ are corrupted by gross outliers. It is straightforward to obtain a full characterization of the stationary points of the population objective using the techniques developed in earlier sections.

    The outline for the paper is as follows. Section~\ref{sec:notation} summarizes notation and basic results we will need. In Section~\ref{sec:alg}, we analyze the linear convergence of the Polyak subgradient method for a class of nonsmooth, nonconvex functions, which includes the subsampled objective $f_S$. In Section~\ref{sec:numerical}, we perform a few proof-of-concept experiments,  illustrating the performance of the Polyak subgradient method on synthetic and real large-scale image recovery problems. Section~\ref{sec:landscape} is devoted to characterizing the nonsmooth landscape of the population objective $f_P$. In Section~\ref{sec:concent_stab}, we develop a concentration theorem for the subdifferential graphs of $f_S$ and $f_P$, and briefly comment on robust extensions. 
%


\section{Notation}\label{sec:notation}

	Throughout, we mostly follow standard notation. The symbol $\R$ will denote the real line, while $\R_+$ and $\R_{++}$ will denote nonnegative and strictly positive real numbers, respectively. 
	We always endow $\R^d$ with the 
	 dot product $\langle x, y \rangle=x^Ty$ and the induced norm $\|x\|:=\sqrt{\langle x,x\rangle}$. 
	The symbol $\mathbb{S}^{d-1}$ will denote the
	unit sphere in $\R^d$, while $B(x,r):=\{y: \|x-y\|<r\}$ will stand for the open ball around $x$ of radius $r>0$.  
	For any set $Q\subset\R^d$, the distance function is defined by  $\dist(x;Q):=\inf_{y\in Q} \|y-x\|.$ The adjoint of a linear map $A\colon\R^d\to\R^m$ will be written as $A^*\colon\R^m\to\R^d$.
	
	Since the main optimization problem we consider is nonsmooth, we will  use some basic generalized derivative constructions. For a more detailed discussion, see for example the monographs of  Mordukhovich~\cite{Mord_1} and Rockafellar-Wets \cite{RW98}.

	Consider a function $f\colon \R^d\to\R$ and a point $\bar x$. The {\em Fr\'{e}chet subdifferential}  of $f$ at $\bar x$, denoted  $\hat{\partial} f(\bar x)$, is the set  of all vectors $v\in\R^d$ satisfying
	$$f(x)\geq f(\bar x)+\langle v,x-\bar x\rangle +o(\|x-\bar x\|)\qquad \textrm{ as }x\to \bar x.$$ 	
	Thus $v$ lies in $\hat{\partial} f(\bar x)$ if and only if the affine function $x\mapsto f(\bar x)+\langle v,x-\bar x\rangle$ minorizes $f$ near $\bar x$ up to first-order. Since the assignment $x\mapsto \hat{\partial} f(x)$ may have poor continuity properties, it is useful to extend the definition slightly. The {\em limiting subdifferential} of $f$ at $\bar x$, denoted $\partial f(\bar x)$, consists of all vectors $v\in \R^d$ such that there exist sequences $x_i$ and $v_i\in \hat\partial f(x_i)$ satisfying $(x_i,f(x_i),v_i)\to(\bar x,f(\bar x),v)$. We say that $\bar x$ is {\em stationary} for $f$ if the inclusion  $0\in \partial f(\bar x)$ holds.
	The {\em graph} of $\partial f$ is the set
	$$\gph \partial f:=\{(x,y)\in \R^d\times\R^d: y\in \partial f(x)\}.$$
	
	For essentially all functions that we will encounter, the two subdifferentials, $\hat{\partial} f(\bar x)$ and ${\partial} f(\bar x)$, coincide. 
	This is the case for $C^1$-smooth functions $f$, where $\hat\partial f(\bar x)$ and $\partial f(\bar x)$ consist only of the gradient $\nabla f(\bar x)$. 
	Similarly for convex function $f$, both subdifferentials reduce to the subdifferential in the sense of convex analysis: 
	$$v\in \partial f(\bar x)\qquad \Longleftrightarrow \qquad f(x)\geq f(\bar x)+\langle v,x-\bar x\rangle\qquad \textrm{ for all }x\in \R^d.$$  
	Most of the nonsmooth functions we will encounter  have a simple composite form: $$F(x):=h(c(x)),$$ where $h\colon\R^m\to\R$ is a finite convex function and $c\colon \R^d\to \R^n$ is a $C^1$-smooth map. For such composite functions, the two subdifferentials coincide, and admit the intuitive chain rule \cite[Theorem 10.6, Corollary 10.9]{RW98}:
	$$\partial F(x)=\nabla c(x)^{*}\partial h(c(x))\qquad \textrm{for all }x\in \R^d.$$

A function $f\colon\R^d\to\R$ is called {\em $\rho$-weakly convex} if $f+\frac{\rho}{2}\|\cdot\|^2$ is a convex function. It follows immediately  from \cite[Theorem 12.17]{RW98} that a lower-semicontinuous function $f$ is $\rho$-weakly convex if and only if the inequality
$$f(y)\geq f(x)+\langle v,y-x\rangle-\frac{\rho}{2}\|y-x\|^2,$$
holds for all points $x,y\in\R^d$ and vectors $v\in \partial f(x)$.

Finally, we will often use implicitly the observation that the Lipschitz constant of any lower-semicontinuous function $f$ on a convex open set $U$ coincides with
$
\displaystyle\sup\{\|\zeta\|: x \in U,\,\zeta \in \partial f(x)\}
$; see e.g. \cite[Theorem 9.13]{RW98}.

%% file: algorithm.tex
\section{Subgradient method}\label{sec:alg}
In this work, we consider the robust formulation of the (real) phase retrieval problem. Setting the stage, suppose we are given vectors $\{a_i\}_{i=1}^m$ in $\R^d$ and measurements $b:=\langle a_i,\bar x\rangle^2$, for a fixed but unknown vector $\bar x$. The goal of the phase retrieval problem is to recover the vector $\bar x\in \R^d$, up to a sign flip. 
The formulation of the problem we consider in this work is:
	$$\min_x~ f_{S}(x):=\frac{1}{m}\sum_{i=1}^m |\langle a_i^T,x\rangle ^2-b_i|.$$
The function $f_S$ (in contrast to other possible formulations) has a number of desirable properties, which we will highlight as we continue.

In this section, we show that the landscape of the  phase retrieval objective $f_S$ favorably lends itself to classical subgradient methods. Namely, with proper initialization and under appropriate statistical assumptions, the  Polyak subgradient method~\cite{poliak}  linearly converges to $\pm x$. 

\subsection{Subgradient method for weakly convex and sharp functions}
The linear convergence guarantees that we present are mostly independent of the structure of $f_S$ and instead rely only on a few general regularity properties, which $f_S$  satisfies under mild statistical assumptions.
Consequently, it will help the exposition in the current section to abstract away from $f_S$.

\begin{assumption}\label{ass:gen_g_assump}{\rm
	Fix a function $g \colon \RR^d \rightarrow \RR$ such that there exist real $\rho, \mu > 0$ satisfying the following two properties.
\begin{enumerate}
\item \textbf{Weak Convexity.} \label{assump:g_convex} The function $g + \frac{\rho}{2}\|\cdot \|^2$ is convex; 
\item \textbf{Sharpness.}\label{assump:g_sharp} The inequality holds: $$g(x)-\min g \geq \mu \cdot \dist(x;\mathcal{X})\qquad \textrm{for all }x \in \R^d,$$
where $\mathcal{X}\neq \emptyset$ is the set of minimizers of $g$.
\end{enumerate}}
\end{assumption}

Duchi and Ruan~\cite{duchi_ruan_PR}, following the work of Eldar-Mendelson \cite{eM}, showed that the  robust phase retrieval loss $f_S(\cdot)$ satisfies Assumption~\ref{ass:gen_g_assump}, under reasonable statistical assumptions. We will discuss these guarantees in Section~\ref{sec:phase_app}, where we will instantiate the subgradient method on the robust phase retrieval objective. Consider now the standard Polyak subgradient method applied to $g$ (Algorithm~\ref{alg:polyak}). 

%

\bigskip

\begin{algorithm}[H]
	\KwData{$x_0 \in \RR^d$}
{\bf Step $k$:} ($k\geq 1$)\\
Choose $\zeta_k \in \partial g(x_k)$.\\

		\eIf{$\zeta_k \neq 0$}{
			Set $x_{k+1}=x_{k} - \frac{g(x_k)-\min g}{\|\zeta_k\|^2}\zeta_k$.
		}{
		Exit algorithm.
	}

\caption{Polyak Subgradient Method}
\label{alg:polyak}
\end{algorithm}
\bigskip
 

%

As the fist step in the analysis of Algorithm~\ref{alg:polyak}, we must ensure that there are no extraneous stationary points 
of $g$ near $\mathcal{X}$. This is the content of the following lemma.

\begin{lem}[Neighborhood with no stationary points]\label{lem:region_without_stat}
	Suppose Assumption~\ref{ass:gen_g_assump} holds. Then $g$ has no stationary points $x$ satisfying 
	\begin{align}\label{eq:region_for_linear}
	0<\dist(x;\mathcal{X})< \frac{2\mu}{\rho}.
	\end{align}
\end{lem}
\begin{proof}
	Consider a stationary point $x$ of $g$, which is outside of $\mathcal{X}$. Let $\bar x\in \mathcal{X}$ be a point satisfying $\|x-\bar x\|=\dist(x;\mathcal{X})$.
 Properties~\ref{assump:g_convex} and ~\ref{assump:g_sharp} then imply
	\begin{align*}
	\mu\cdot \dist(x; \mathcal{X}) \leq g(x)-g(\bar x) \leq  \frac{\rho}{2}\|x - \bar x\|^2=\frac{\rho}{2}\cdot\dist^2(x; \mathcal{X}).
	\end{align*}
	Dividing through by $\dist(x; \mathcal{X})$, the result follows.
\end{proof}

The following Theorem~\ref{thm:qlinear} -- the main result of this subsection -- shows that when Algorithm~\ref{alg:polyak} is initialized within a certain tube $\mathcal{T}$ of $\mathcal{X}$, the iterates $x_k$ stay within the tube and converge linearly to $\mathcal{X}$. It is interesting to note that the rate of local linear  convergence does not depend on the weak convexity constant $\rho$; indeed, the value $\rho$  only dictates the size of the tube $\mathcal{T}$.

\begin{thm}[Linear rate]\label{thm:qlinear}
	Suppose Assumption~\ref{ass:gen_g_assump} holds. Fix a real $\gamma \in (0,1)$ and define  the tube 
	$$\mathcal{T}:=\left\{x\in \R^d: \dist(x;\mathcal{X})\leq \gamma\cdot\frac{\mu}{\rho}\right\},$$
	and the corresponding Lipschitz constant 
	$$\displaystyle L_g:=\sup_{x\in \mathcal{T},\,\zeta\in \partial g(x) } \|\zeta\|.$$
	Then  Algorithm~\ref{alg:polyak} initialized at any point $x_0\in \mathcal{T}$ produces iterates that 
	converge $Q$-linearly to $\mathcal{X}$ at the rate: 
	\begin{align}\label{eqn:lin_rate1}
	\dist^2(x_{k+1};\mathcal{X}) \leq \left(1-\frac{(1-\gamma) \mu ^2}{L_g^2}\right)\dist^2(x_{k};\mathcal{X}),
	\end{align}
\end{thm}
\begin{proof}
	We proceed by induction. Suppose that the theorem holds up to iteration $k$.  We will prove the inequality \eqref{eqn:lin_rate1}.
	To this end, let $\bar x\in \mathcal{X}$ be a point satisfying $\|x_k-\bar x\|=\dist(x_k;\mathcal{X})$. Note that if $x_k$ lies in $\mathcal{X}$, there is nothing to prove. Thus we may suppose $x_k\notin \mathcal{X}$.
	Note that the inductive hypothesis implies $\dist(x_{k};S)\leq \dist(x_{0};S)$ and therefore $x_k$ lies in $\mathcal{T}$. Lemma~\ref{thm:qlinear} therefore guarantees
 $\zeta_k\neq 0$.
	Using Properties~\ref{assump:g_convex} and~\ref{assump:g_sharp}, we successively deduce 
	\begin{align*}
	\|x_{k+1} - \bar x\|^2 
	&= \|x_{k} - \bar x\|^2 + 2\dotp{x_{k} - \bar x, x_{k+1} - x_k} + \|x_{k+1} - x_k\|^2 \\
	&= \|x_{k} - \bar x\|^2 + \frac{2(g(x_k) - g(\bar x))}{\|\zeta_k\|^2}\cdot\dotp{\zeta_k,\bar x - x_{k}} + \frac{(g(x_k) - g(\bar x))^2}{\|\zeta_k\|^2}\\
	&\leq \|x_{k} - \bar x\|^2 + \frac{2(g(x_k) - g(\bar x))}{\|\zeta_k\|^2}\left( g(\bar x) - g(x_k) + \frac{\rho}{2}\|x_k - \bar x\|^2 \right)\ + \frac{(g(x_k) - g(\bar x))^2}{\|\zeta_k\|^2} \\
	&= \|x_{k} - \bar x\|^2 + \frac{(g(x_k) - g(\bar x))}{\|\zeta_k\|^2}\left(\rho\|x_k - \bar x\|^2 -  (g(x_k) - g(\bar x))  \right)\\
	&\leq \|x_{k} - \bar x\|^2 + \frac{(g(x_k) - g(\bar x))}{\|\zeta_k\|^2}\left(\rho\|x_k - \bar x\|^2 -  \mu\|x_k - \bar x\|  \right) \\
	&= \|x_{k} - \bar x\|^2 + \frac{\rho (g(x_k) - g(\bar x))}{\|\zeta_k\|^2}\left( \|x_k - \bar x\| -\frac{\mu}{\rho}\right)\|x_k - \bar x\|.
	\end{align*}
	Combining the inclusion $x_k\in \mathcal{T}$ with sharpness (Assumption~\ref{assump:g_sharp}), we therefore deduce
	$$\dist^2(x_{k+1};\mathcal{X})\leq \|x_{k+1} - \bar x\|^2\leq \left(1-\frac{(1-\gamma) \mu ^2}{\|\zeta_k\|^2}\right)\|x_k-\bar x\|^2.$$
	The result follows.
\end{proof}

\subsection{Convergence for the phase retrieval objective}\label{sec:phase_app}
We now turn to an application of Theorem~\ref{thm:qlinear} to the phase retrieval loss $f_S$. In particular, to run the subgradient method, we must only compute a subgradient of $f_S$, which can be easily done using the chain rule:
 $$
 \frac{1}{m}\sum_{i=1}^m 2\dotp{a_i, x}\cdot\sign(\langle a_i, x\rangle^2 - b_i)a_i \in \partial f_S(x).
 $$
Each iteration of Algorithm~\ref{alg:polyak} thus requires a single pass through the set of measurement vectors. We will see momentarily that under mild statistical assumptions, $\{\pm\bar x\}$ are the unique minimizers of $f_S$, as soon as $m>2d$.

Thus for a successful application of Theorem~\ref{thm:qlinear}, we must only address the following questions:
\begin{enumerate}[{(\em i)}]
	\item Describe the statistical conditions on the data generating mechanism, which insure that Assumption~\ref{ass:gen_g_assump}  holds with high probability.
	\item Estimate the Lipschitz constant of $f_S$ on the union of balls $\mathcal{T}=B(\bar x, \frac{\gamma\mu}{\rho})\cup B(-\bar x, \frac{\gamma\mu}{\rho})$.
	\item Describe a good initialization procedure for producing $x_0\in \mathcal{T}$.
\end{enumerate}
Essentially all of these points follow from the work of Duchi and Ruan~\cite{duchi_ruan_PR}, Eldar-Mendelson \cite{eM}, and Wang et al. \cite{eldar_init}. We summarize them here for the sake of completeness. Henceforth, let us suppose that $a_i\in \R^d$ (for $i=1,\ldots,m$) are independent realizations of a random vector $a\in \R^d$. 

\subsubsection{Sharpness}
In order to ensure sharpness (or rather the stronger  ``stability'' property \cite{eM}), we make the following assumption on the distribution of $a$.

\begin{assumption}\label{ass:weird_prob}
	{\rm 
		 There exist constants $\kappa_{\mathrm{st}}^*,p_0 > 0$ such that for all $u, v \in \mathbb{S}^{d-1}$, we have
			$$
			\mathbb{P}\left( |\dotp{a, v}\dotp{a,u}| \geq \kappa_{\mathrm{st}}^\ast \right) \geq  p_0,
			$$
}
\end{assumption}

Roughly speaking, this mild assumption simply says that the random vector $a$ has sufficient support in all directions.
In particular, the standard Gaussian $a\sim \mathsf{N}(0,I_d)$ satisfies Assumption~\ref{ass:weird_prob} with $\kappa_{\mathrm{st}}^*=0.365$ and $p_0=0.25$; see \cite[Example 1]{duchi_ruan_PR}. The following is proved in \cite[Corollary 3.1]{duchi_ruan_PR}.

\begin{thm}[Sharpness]\label{thm:sharp_duchi}
	Suppose that Assumption~\ref{ass:weird_prob} holds. Then there exists a numerical constant $c < \infty$ such that if $ mp_0^2 \geq cd$, we have 
	$$
	\mathbb{P}\left(f_S(x) - f_S(\bar x)  \geq \frac{1}{2}\kappa_{\mathrm{st}}^\ast p_0\|x - \bar x\| \|x + \bar x\| \quad \textrm{for all } x \in \RR^d \right) \geq 1 - 2\exp\left(-\frac{mp_0^2}{32}\right).
	$$
\end{thm}

To simplify notation, set $\dist(x;\bar x):=\min\{\|x-\bar x\|,\|x+\bar x\|\}$.
Thus Assumption~\ref{ass:weird_prob}  implies, with high probability, that $f_S$ is sharp. Indeed, Theorem~\ref{thm:sharp_duchi} directly implies that with high probability we have 
\begin{align*}
f_S(x) - f_S(\bar x)&\geq  \tfrac{1}{2}\kappa_{\mathrm{st}}^\ast p_0\|x - \bar x\| \|x + \bar x\|\\
&\geq \tfrac{1}{2}\kappa_{\mathrm{st}}^\ast p_0\cdot\min\{\|x-\bar x\|,\|x+\bar x\|\}\cdot\max\{\|x-\bar x\|,\|x+\bar x\|\} \\
&\geq  \tfrac{1}{2}\kappa_{\mathrm{st}}^\ast p_0 \|\bar x\|\cdot \dist(x;\bar x).
\end{align*}
Thus the sharpness condition in Assumption~\ref{ass:gen_g_assump} holds for $g=f_S$ with $\mu=\frac{1}{2}\kappa_{\mathrm{st}}^\ast p_0 \|\bar x\|$. 

\subsubsection{Weak convexity}
We next look at weak convexity of the objective $f_S$. We will need the following definition.
\begin{definition}
	{\rm 	A random vector $a\in\R^d$ is {\em $\sigma^2$-sub-Gaussian} if for all unit vectors $v\in \mathbb{S}^{d-1}$, we have 
		$$\mathbb{E}\left[\exp\left(\frac{\langle a,v\rangle^2}{\sigma^2}\right)\right]\leq e.$$		
	}
\end{definition}

\begin{assumption}\label{ass:subgauss}
	{\rm 
		The random vector $a$ is $\sigma^2$-sub-Gaussian.
	}
\end{assumption}

The following is a direct consequence of \cite[Corollary 3.2]{duchi_ruan_PR}.

\begin{thm}[Weak convexity]\label{lem:weak_conv}
	Suppose that Assumption~\ref{ass:subgauss} holds.  Then there exists a numerical constant $c<\infty$ such that whenever $m\geq cd$, the function $f_S$ is $4\sigma^2$-weakly convex, with probability at least $1-\exp\left(-\frac{m}{c}\right)$.  
\end{thm}
\begin{proof}
	This follows almost immediately from \cite[Corollary 3.2]{duchi_ruan_PR}. 
	Define the separable function $h(z_1,\ldots, z_m):=\frac{1}{m}\sum_{i=1}^m |z_i|$ and the map $F\colon\R^d\to\R^m$ with the $i$'th coordinate given by  $F_i(x):=(a^T_ix)^2-b_i$. Observe the equality $f_S(x)=h(F(x))$.
	Corollary 3.2 in \cite{duchi_ruan_PR} shows that there exists a numerical constant $c<\infty$ such that whenever $m\geq cd$, with probability at least $1-\exp\left(-\frac{m}{c}\right)$, we have 
	$$f_S(y)\geq h(F(x)+\nabla F(x)(y-x))-2\sigma^2\|y-x\|^2\qquad \textrm{for all }x,y\in \R^d.$$
	Since $h$ is convex, for any vector $v\in \partial h(F(x))$ we have
	$$h(F(x)+\nabla F(x)(y-x))\geq h(F(x))+\langle v,\nabla F(x)(y-x)\rangle=f_S(x)+ \langle \nabla F(x)^*v,y-x\rangle.$$
	Taking into account the equality $\partial f_S(x)=\nabla F(x)^*\partial h(F(x))$, we conclude that $f_S$ is $4\sigma^2$-weakly convex.
\end{proof}

\subsubsection{Lipschitz constant on a ball}
Let us next estimate the Lipschitz constant of $f_S$ on a ball of a fixed radius. To this end, observe the chain of inequalities
\begin{equation}
\begin{aligned}\label{eqn:main_ineq_lip}
		 |f_S(x)-f_S(y)|&\leq \frac{1}{m}\sum_{i=1}^m \left||\langle a_i,x\rangle^2-\langle a_i,\bar x\rangle^2 |- |\langle a_i,y\rangle^2-\langle a_i,\bar x\rangle^2|\right|\\
		 &\leq \frac{1}{m}\sum_{i=1}^m |\langle a_i,x\rangle^2-\langle a_i,y\rangle^2|\\
		 &=\|x-y\|\|x+y\|\cdot \frac{1}{m}\sum_{i=1}^m |\langle a_i,v\rangle\langle a_i,w\rangle|,
		 \end{aligned}
		 \end{equation}
where we set $v:=\frac{x-y}{\|x-y\|}$ and $w:=\frac{x+y}{\|x+y\|}$. Thus we would like to upper-bound the term $\frac{1}{m}\sum_{i=1}^m |\langle a_i,v\rangle\langle a_i,w\rangle|$ by a numerical constant, with high probability. Intuitively, there are two key ingredients that would ensure this bound: the random vector $a\in\R^d$ should have light tails (sub-Gaussian) and $a$ should not concentrate too much along any single direction. A standard way to model the latter is through an isotropy assumption.

\begin{definition}[Isotropy]
{\rm
A random vector $a\in \R^d$ is {\em isotropic} if $\mathbb{E}[aa^T]=I_d$.
}
\end{definition}

Note that $a\in \R^d$ is isotropic if and only if $\mathbb{E}[\langle a,v\rangle^2]=1$ for all unit vectors $v\in \mathbb{S}^{d-1}$.

\begin{assumption}\label{ass:isotropy}
	{\rm 
		The random vector $a$ is isotropic.
	}
\end{assumption}

 Assumptions~\ref{ass:subgauss} and \ref{ass:isotropy} imply that the term $\frac{1}{m}\sum_{i=1}^m |\langle a_i,v\rangle\langle a_i,w\rangle|$ cannot deviate too much from its mean, uniformly over all unit vectors $v,w\in \R^d$. Indeed, the following is a special case of \cite[Theorem 2.8]{eM}.

\begin{thm}[Concentration]\label{thm:em}
	Suppose that Assumptions~\ref{ass:subgauss} and \ref{ass:isotropy} hold.  Then there exist constants  $c_1,c_2,c_3$ depending only on $\sigma$ so that with probability at least
	$1-2\exp(-c_2c_1^2\min\{m,d^2\}),$
	the inequality holds:
	  $$\sup_{v,w\in \mathbb{S}^{d-1}}\left|\frac{1}{m}\sum_{i=1}^m |\langle a_i,v\rangle\langle a_i,w\rangle|-\mathbb{E}_a[|\langle a,v \rangle \langle a,w\rangle|]\right|\leq c_1^3c_3\left( \sqrt{\frac{d}{m}}+\frac{d}{m}\right).$$	\end{thm}

We can now establish Lipschitz behavior of $f_S$ on bounded sets.

\begin{cor}[Lipschitz constant on a ball]
Suppose that Assumptions~\ref{ass:subgauss} and \ref{ass:isotropy} hold. Then there exist constants $c_1,c_2,c_3$ depending only on $\sigma$ such that with probability at least
		  $$1-2\exp(-c_2c_1^2\min\{m,d^2\}),$$
we have 
\begin{equation}\label{eqn:lip}
|f_S(x)-f_S(y)|\leq \left(1+c_1^3c_3\left( \sqrt{\frac{d}{m}}+\frac{d}{m}\right)\right)\|x-y\|\|x+y\|\qquad \textrm{ for all }x,y\in \R^d,
\end{equation}
and consequently
\begin{equation}\label{eqn:subgrad_norm}
	\max_{\zeta\in \partial f_S(x)} \|\zeta\|  \leq 2\left(1+c_1^3c_3\cdot \left( \sqrt{\frac{d}{m}}+\frac{d}{m}\right)\right) \|x\|\qquad \textrm{ for all }x\in \R^d.
\end{equation}
\end{cor}
\begin{proof}
	Combining inequalities \eqref{eqn:main_ineq_lip} with Theorem~\ref{thm:em}, we deduce that there exist constants $c_1,c_2,c_3$ depending only on $\sigma$ such that with probability 
		  $$1-2\exp(-c_2c_1^2\min\{m,d^2\}),$$ all points $x,y\in \R^d$ satisfy
$$|f_S(x)-f_S(y)|\leq \left(\mathbb{E}_a[|\langle a,v \rangle \langle a,w\rangle|]+c_1^3c_3\left( \sqrt{\frac{d}{m}}+\frac{d}{m}\right)\right)\|x-y\|\|x+y\|,$$
where we set $v:=\frac{x-y}{\|x-y\|}$ and $w:=\frac{x+y}{\|x+y\|}$.
	Isotropy, in turn, implies 		
	  \begin{align*}
		  \mathbb{E}_a[|\langle a,v\rangle \langle a,w\rangle|]&\leq \sqrt{\mathbb{E}_a[|\langle a,v \rangle|^2]}\cdot\sqrt{\mathbb{E}_a[|\langle a,w \rangle|^2]}= 1,
		  \end{align*}
Equation~\eqref{eqn:lip} follows immediately. Consequently, notice
$$\ls_{x,y\to z}~\frac{|f_S(x)-f_S(y)|}{\|x-y\|}\leq 2\left(1+c_1^3c_3\left( \sqrt{\frac{d}{m}}+\frac{d}{m}\right)\right)\|z\|.$$
Since the Lipschitz constant of $f_S$ at $x$ coincides with the value $\max_{\zeta\in \partial f(x)}\|\zeta\|$ (see e.g. \cite[Theorem 9.13]{RW98}), the estimate \eqref{eqn:subgrad_norm} follows.
\end{proof}

We now have all the ingredients in place to apply 
Theorem~\ref{thm:qlinear} to the robust phase retrieval objective. Namely, under Assumptions~\ref{ass:weird_prob}, \ref{ass:subgauss}, and \ref{ass:isotropy}, we may set\footnote{The definition of $L_g$ uses that the norm of any point in the tube $\mathcal{T}=B(\bar x, \frac{\gamma\mu}{\rho})\cup B(-\bar x, \frac{\gamma\mu}{\rho})$ is clearly upper bounded by $\|\bar x\|+\frac{\mu}{2\rho}$.}
\begin{align}\label{eq:phase_constants}
\rho := 4\sigma^2; && \mu := \tfrac{1}{2}\kappa_{\mathrm{st}}^\ast p_0\|\bar x\|; && L_g := 2\left(1+c_1^3c_3\cdot \left( \sqrt{\frac{d}{m}}+\frac{d}{m}\right)\right) \left(1 +  \frac{\kappa_{\mathrm{st}}^\ast p_0}{16\sigma^2}\right)\|\bar x \|.
\end{align}

Thus, we have proved the following convergence guarantee -- the main result of this section. To simplify the formulas, we apply Theorem~\ref{thm:qlinear} only with $\gamma:=1/2$.

\begin{cor}[Linear convergence for phase retrieval]
Suppose	that Assumptions~\ref{ass:weird_prob}, \ref{ass:subgauss}, and \ref{ass:isotropy} hold. 
Then there exists a numerical constant $c<\infty$ 
 such that the following is true. Whenever we are in the regime,
 $\frac{c}{p_0^2}\leq \frac{m}{d}\leq d$, and we initialize  Algorithm~\ref{alg:polyak} at  $x_0$ satisfying
\begin{equation}\label{eqn:init_ball2}
\min\left\{\frac{\|x_0-\bar x\|}{\|\bar x\|},\frac{\|x_0+\bar x\|}{\|\bar x\|}\right\}\leq \frac{\kappa_{\mathrm{st}}^\ast p_0}{16\sigma^2},
\end{equation}	
 we can be sure with  probability at least 
$$1-6\exp\left(-m\cdot \min\left\{\tfrac{p_0^2}{32},c^{-1},\tilde{c}\right\}\right)$$ 
that the produced iterates $\{x_k\}$ converge  to $\{\pm\bar x\}$ at the linear rate:
\begin{align}
\dist^2(x_{k+1};\bar x) \leq \left(1  - \left(\frac{p_0 \kappa_{\mathrm{st}}^\ast}{\sqrt{32}\left(1+\hat c\cdot \left( \sqrt{\frac{p_0^2}{c}}+\frac{p_0^2}{c}\right)\right) \left(1 +  \frac{\kappa_{\mathrm{st}}^\ast p_0}{16\sigma^2}\right)}\right)^2\right)\dist^2(x_{k};\bar x).
\end{align}
Here,  $\tilde{c}$ and $\hat c$ are constants that depend only on $\sigma$. In particular, aside from numerical constants, the linear rate depends only on $\kappa_{\mathrm{st}}^\ast$, $p_0$, and $\sigma$.
\end{cor}

%

Thus under typical statistical assumptions, the subgradient method converges linearly to $\{\pm\bar x\}$, as long as one can initialize the method at a point $x_0$  satisfying the relative error condition $\|x_0\pm\bar x\|\leq  R{\|\bar x\|}$, where $R$ is a constant. A number of authors have proposed initialization strategies that can achieve this guarantee using only a constant multiple of $d$ measurements \cite{versh_kac,eldar_init,duchi_ruan_PR,wirt_flow,whitflow2}. For completeness, we record the strategy that was proposed in \cite{eldar_init}, and rigorously justified in \cite{duchi_ruan_PR}.
To simplify the exposition, we only state the guarantees of the initialization under Gaussian assumptions on the measurement vectors $a_i$.

\begin{thm}[{\cite[Equation~(15)]{duchi_ruan_PR}}]\label{thm:init}
Assume that $a_i \sim \mathsf{N}(0, I_d)$ are i.i.d. standard Gaussian. Define the value $\hat r^2 := \frac{1}{m}\sum_{i=1}^m b_i$ and the index set $\mathcal{I}_{\mathrm{sel}} := \{i \in [m] \mid b_i \leq \frac{1}{2} \hat r^2\} $. Set  
\begin{align*}
X^{\mathrm{init}} := \sum_{i \in \mathcal{I}_{\mathrm{sel}}} a_ia_i^T \qquad \textrm{and}\qquad \hat w := \argmin_{w \in \mathbb{S}^{d-1}} \; w^T \mathrm{X}^{\mathrm{init}} w.
\end{align*}
Then as soon as $\frac{m}{d} \gtrsim \varepsilon^{-2}$ the point $x_0 = \hat r \hat w$ satisfies
\begin{align*}
\min\left\{\frac{\|x_0-\bar x\|}{\|\bar x\|},\frac{\|x_0+\bar x\|}{\|\bar x\|}\right\} \lesssim \varepsilon \log\frac{1}{\varepsilon}
\end{align*}
with probability at least $\geq 1 - 5 \exp(-c m \varepsilon^2)$, where $c$ is a numerical constant.
\end{thm}

For more details and intuition underlying the initialization procedure, see~\cite[Section 3.3]{duchi_ruan_PR}. 


%% file: numerics.tex
\section{Numerical Illustration}\label{sec:numerical}	
In this section, as a proof of concept, we apply the subgradient method to medium and large-scale phase retrieval problems. All of our experiments were performed on a standard desktop:  Intel(R) Core(TM)  i7-4770 CPU\@3.40 GHz with 8.00 GB RAM.

We begin with simulated data. Set $d=5000$. We generated a standard Gaussian random matrix $A\in \R^{m\times d}$ for each  value $m\in \{11,000,12225, 13500, 14750, 16000, 17250,18500\}$; afterwards, we generated a Gaussian vector $\bar x\sim\mathsf{N}(0,I_d)$ and  set $b=(A\bar x)^2$. We then applied the initialization procedure, detailed in Theorem~\ref{thm:init}, followed by the subgradient method. Figure~\ref{fig:sim_data} plots the progress of the iterates produced by the subgradient method in each of the seven experiments. The top curve corresponds to $m=11,000$, the bottom curve corresponds to $m=18500$, while the curves for the other values of $m$ interpolate in between. The iterates corresponding to $m=11,000$ stagnate; evidently the number of measurements is too small. Indeed, the iterates do not even converge to a stationary point of the problem; this is in contrast to the prox-linear method in \cite{duchi_ruan_PR}. The iterates for the rest of the experiments converge to the true signal $\pm\bar x$ at an impressive linear rate.

In out second experiment, we use digit images from the MNIST data set \cite{MNIST}; these are relatively small so that the measurement matrices can be stored in memory.  We illustrate the generic behavior of the algorithm on digit seven in Figure~\ref{fig:mnist}. The dimensions of the image we use are $32\times 32$ (with 3 RGB channels).
Hence, after vectorizing the dimension of the variable is $d=3072$, while the number of Gaussian  measurements is $m=3d=9216$.
The initialization produced appears to be reasonable; the digit is visually discernible. The true image and the final image produced by the method are essentially identical. The convergence plot appears in Figure~\ref{fig:7_error}. 

\begin{figure}[!h]
	\centering
	\begin{minipage}[b]{0.32\textwidth}
		\includegraphics[width=\textwidth]{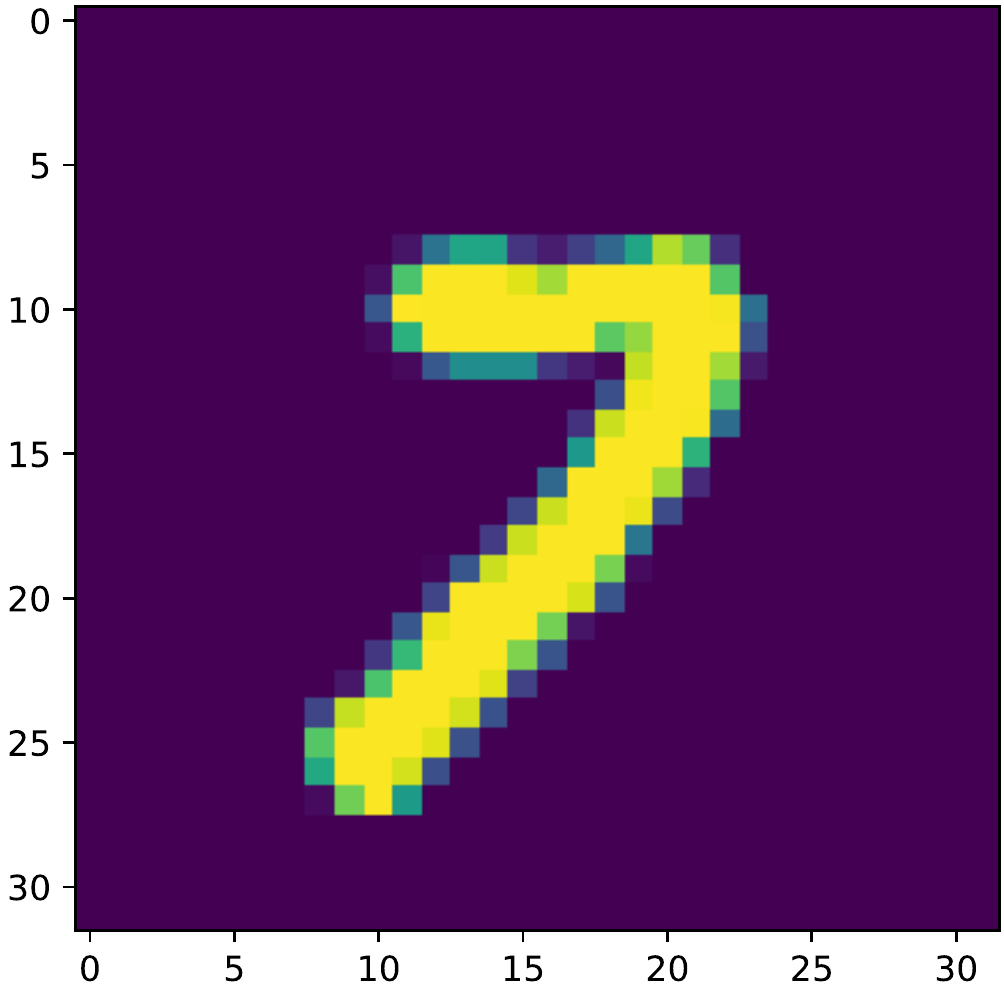}
	\end{minipage}
	\begin{minipage}[b]{0.32\textwidth}
		\includegraphics[width=\textwidth]{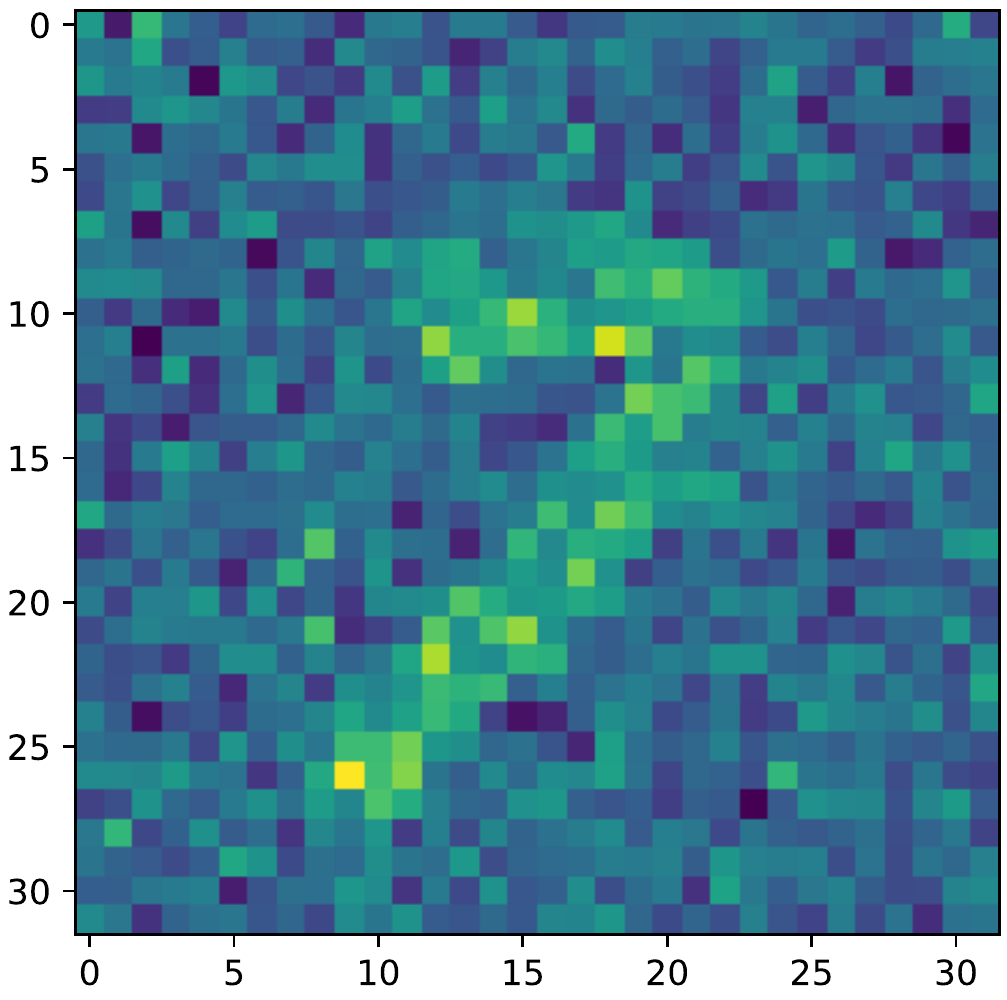}
	\end{minipage}
	\begin{minipage}[b]{0.32\textwidth}
		\includegraphics[width=\textwidth]{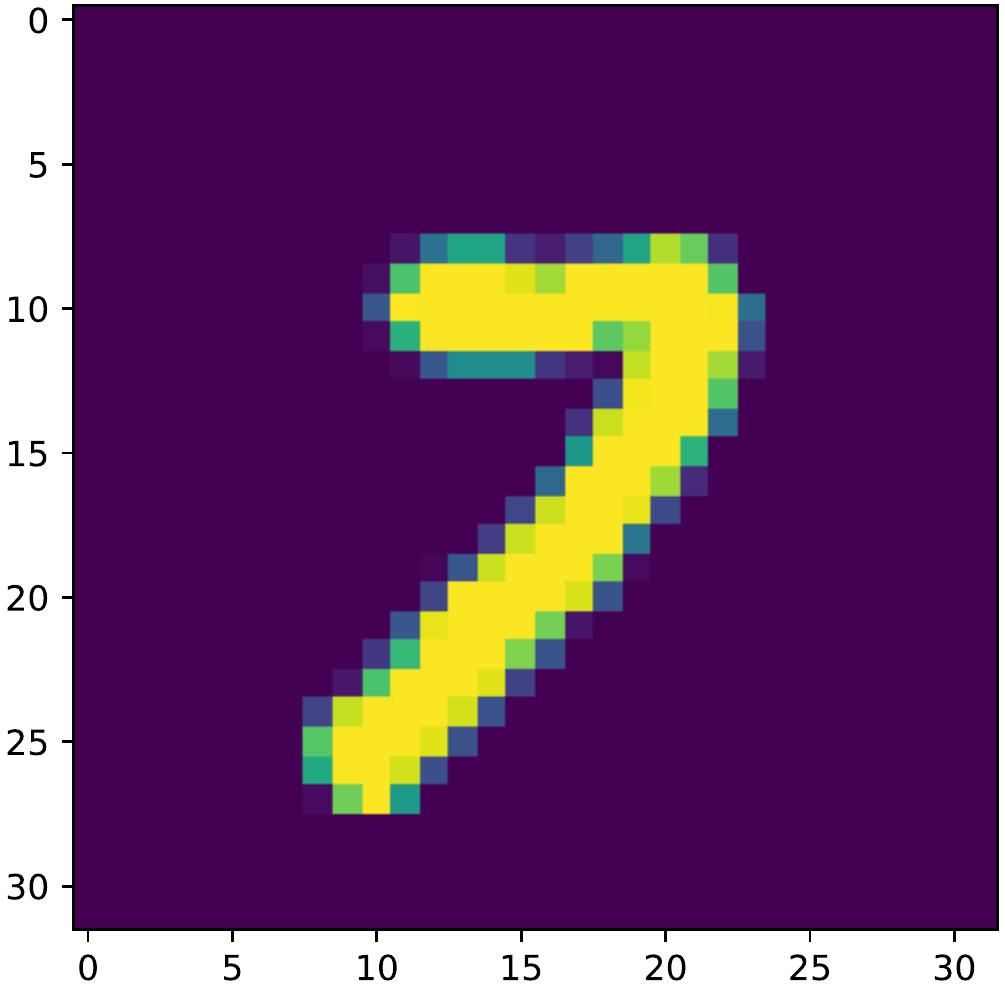}
	\end{minipage}
	\caption{Digit recovery; left is the true digit, middle is the initial, right is the digit produced by the subgradient method. Dimension of the problem: $(n,d,m)=(32,3072,9216)$.}
	\label{fig:mnist}
\end{figure}

\begin{figure}[!h]
	\centering
	\includegraphics[scale=0.7]{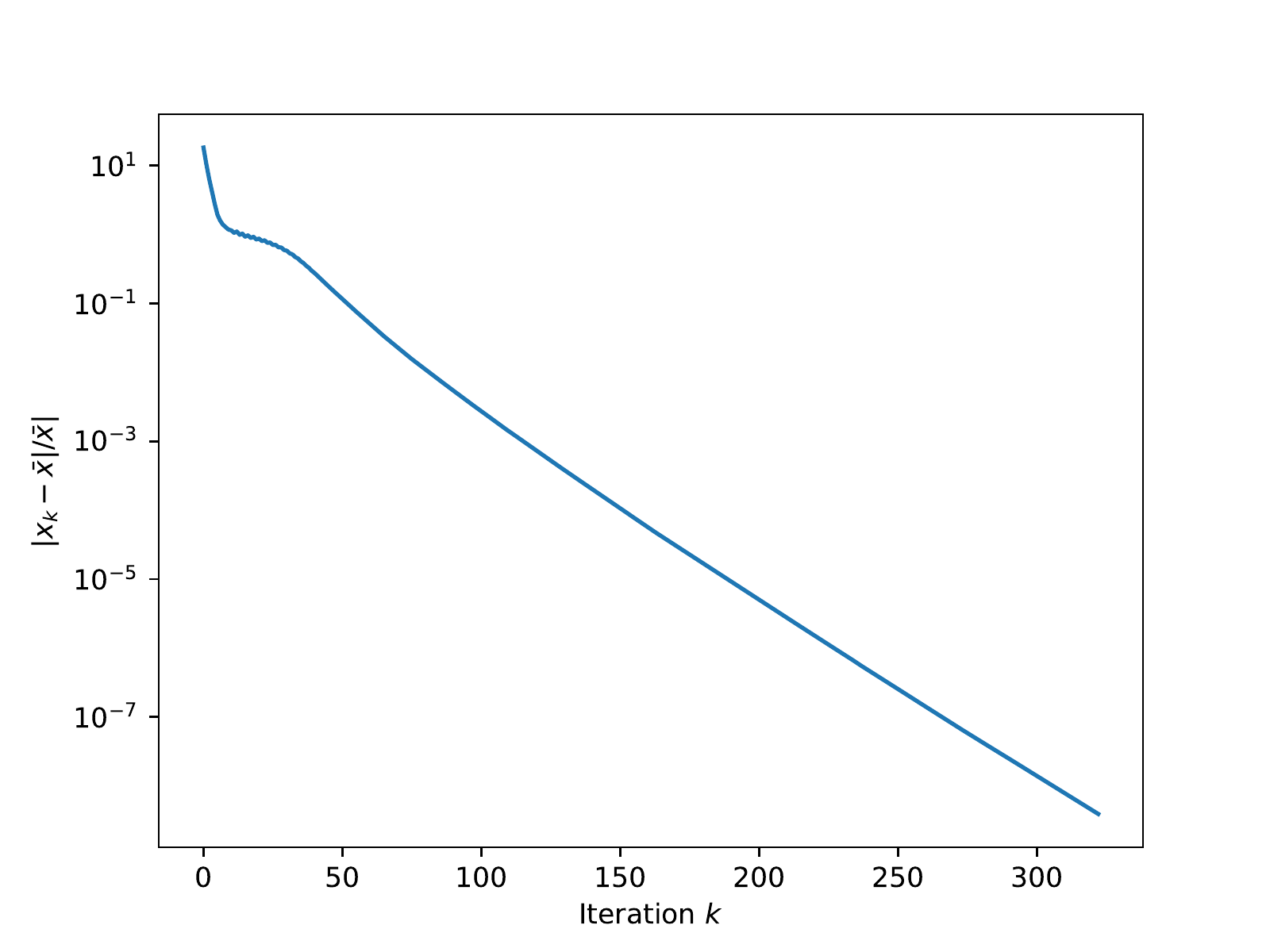}
	\caption{Convergence plot on MNIST digit (iterates vs. $\|x_k-\bar x\|/\|\bar x\|$).}
	\label{fig:7_error}
\end{figure}

	\begin{figure}[!h]
		\centering
		\includegraphics[scale=0.7]{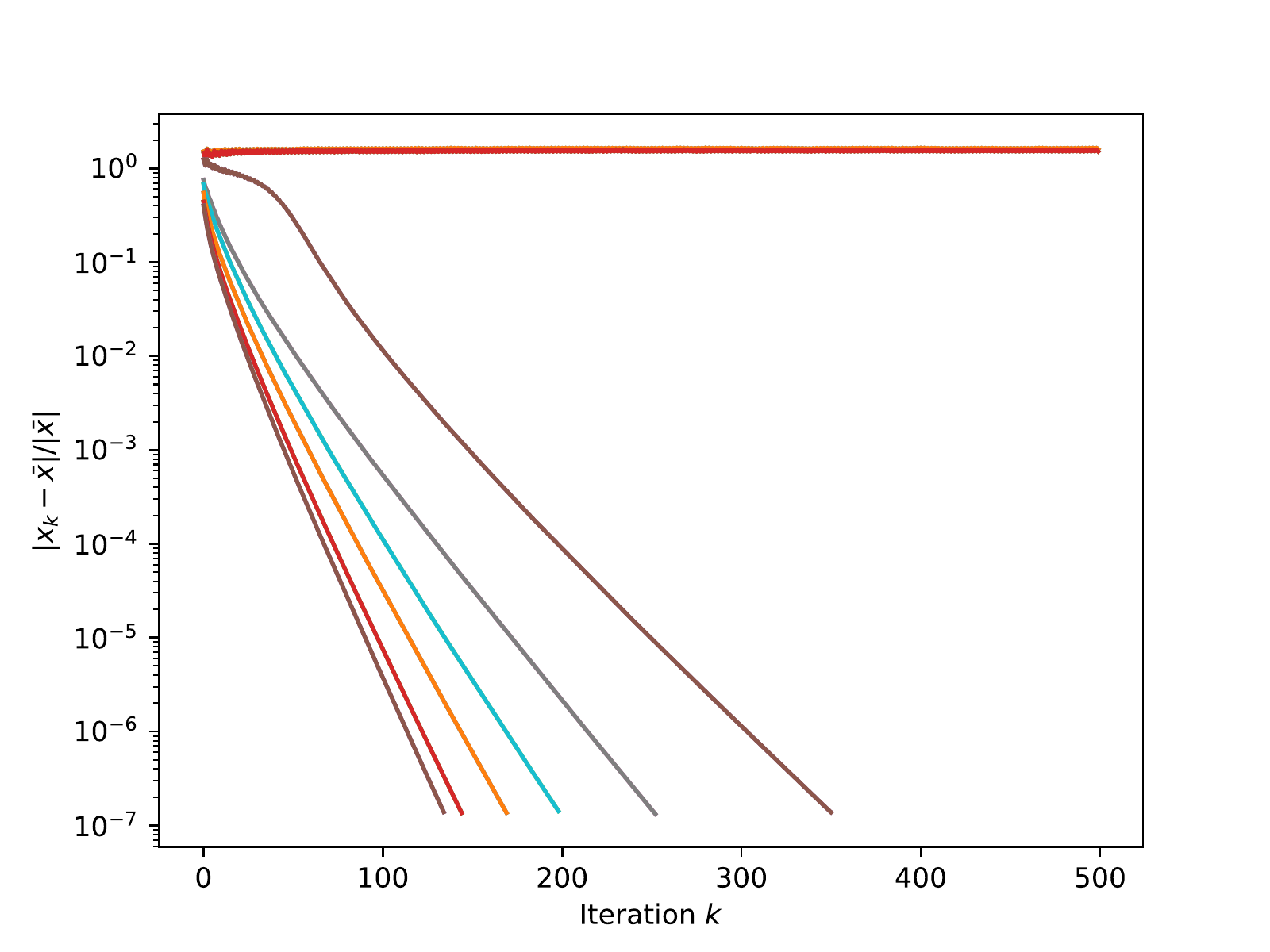}
		\caption{Convergence plot for the experiment on simulated data (iteration vs. $\|x_k-\bar x\|/\|\bar x\|$).}\label{fig:sim_data}
	\end{figure}

We next apply the subgradient method for recovering large-scale real images. To allow an easy comparison with previous work, we generate the data 
 using the same process as in \cite[Section 6.3]{duchi_ruan_PR}. We first describe how we generate the operator $A$. To this end, let $H\in \{-1,1\}^{l\times l}/\sqrt{l}$ be a symmetric normalized Hadamard matrix. Consequently $H$  satisfies the equation $H^2=I_l$. Note that by the virtue of being Hadamard, matrix vector multiplication $Hv$ requires time $l\log(l)$.
 For some integer $k$, we then generate $k$ i.i.d. diagonal sign matrices $S_1,\ldots, S_k\in \diag(\{-1,1\}^l)$ uniformly at random, and define $A=\begin{bmatrix} H S_1&H S_2&\ldots& H S_k \end{bmatrix}^T\in \R^{kl\times l}$.
 
  We work with square colored images, represented as an array $\overline{X}\in \R^{n\times n\times 3}$. The number $3$ appears because colored images have $3$ RGB channels. We then stretch the matrix 
   $\overline{X}$ into a $3n^2$-dimensional vector $\bar x$ and set the measurements $b_i:=(A(i,\cdot)\bar x)^2$, where $A(i,\cdot)$ denotes the $i$'th row of $A$. Thus if the image is $n\times n$, the number of variables in the problem formulation is $d:=3n^2$ and the number of measurements is $m:=kd=3kn^2$. We  use the initialization procedure proposed in Theorem~\ref{thm:init}, with a standard power method (with a shift) to find the minimal eigenvalue of $X^{\textrm{init}}$. We complete the experiment by running the subgradient method (Algorithm~\ref{alg:polyak}), which requires no parameter tunning.


We  perform a large scale experiment on two pictures taken by the Hubble telescope. Figure~\ref{fig:hubbl_rec_image} describes the results of the experiment, while Figure~\ref{fig:error_hubble} plots the iterate progress.
The image on the left is $1024\times 1024$ and we use $k=3$ Hadamard matrices. Hence the dimensions of the problem are $d\approx 2^{22}$ and $m=3d\approx 2^{24}$. The image on the right is $2048\times 2048$ and we use $k=3$ Hadamard matrices. Hence the dimensions of the problem are $d\approx 2^{24}$ and $m=3d\approx2^{25}$. 
 For the image on the left, the entire experiment, including initialization and the subgradient method completed in 3 min. For the image on the right, it completed in $25.6$ min. The vast majority of time was taken up by the initialization. Thus a more careful implementation and/or tunning of the initialization procedure could speed up the experiment. 

		\begin{figure}[!h]
	\centering
		\begin{minipage}[b]{0.49\textwidth}
		\includegraphics[width=\textwidth]{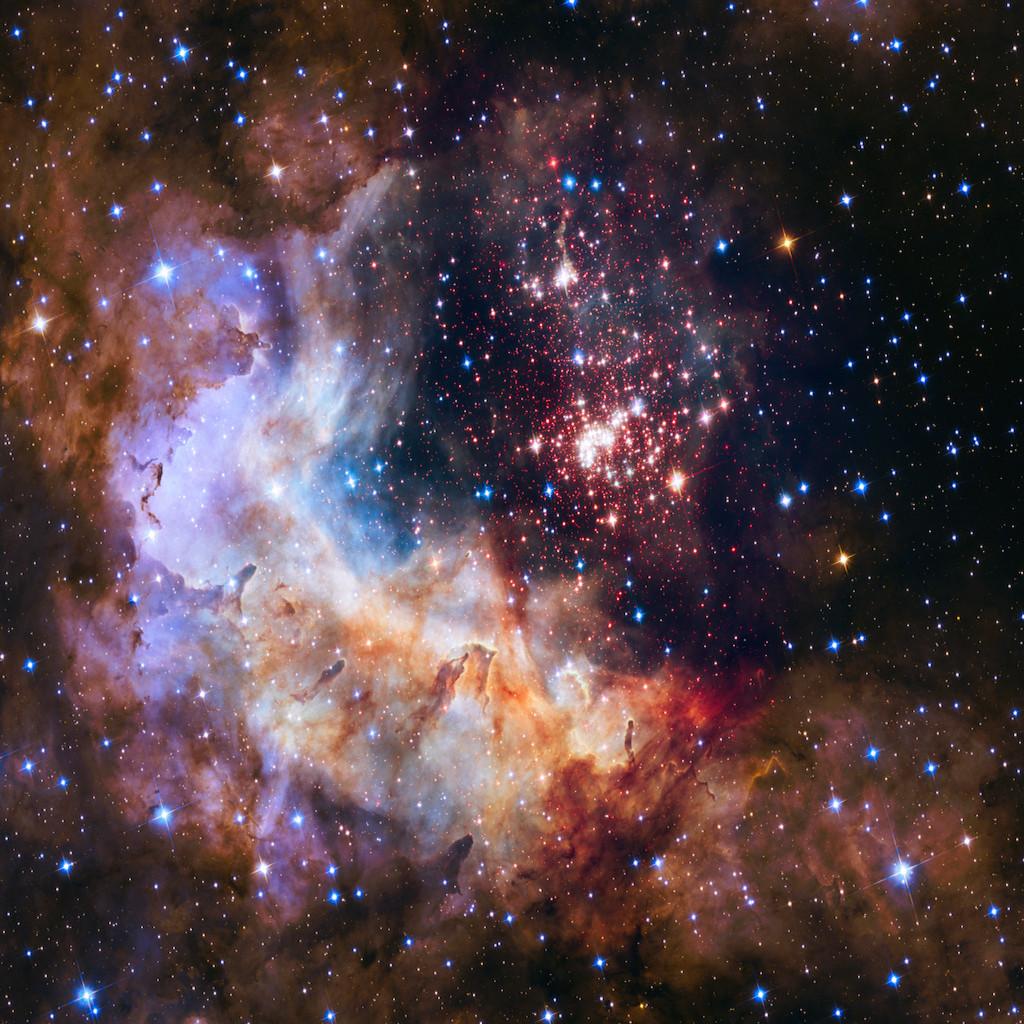}
		\end{minipage}~
		\begin{minipage}[b]{0.49\textwidth}
		\includegraphics[width=\textwidth]{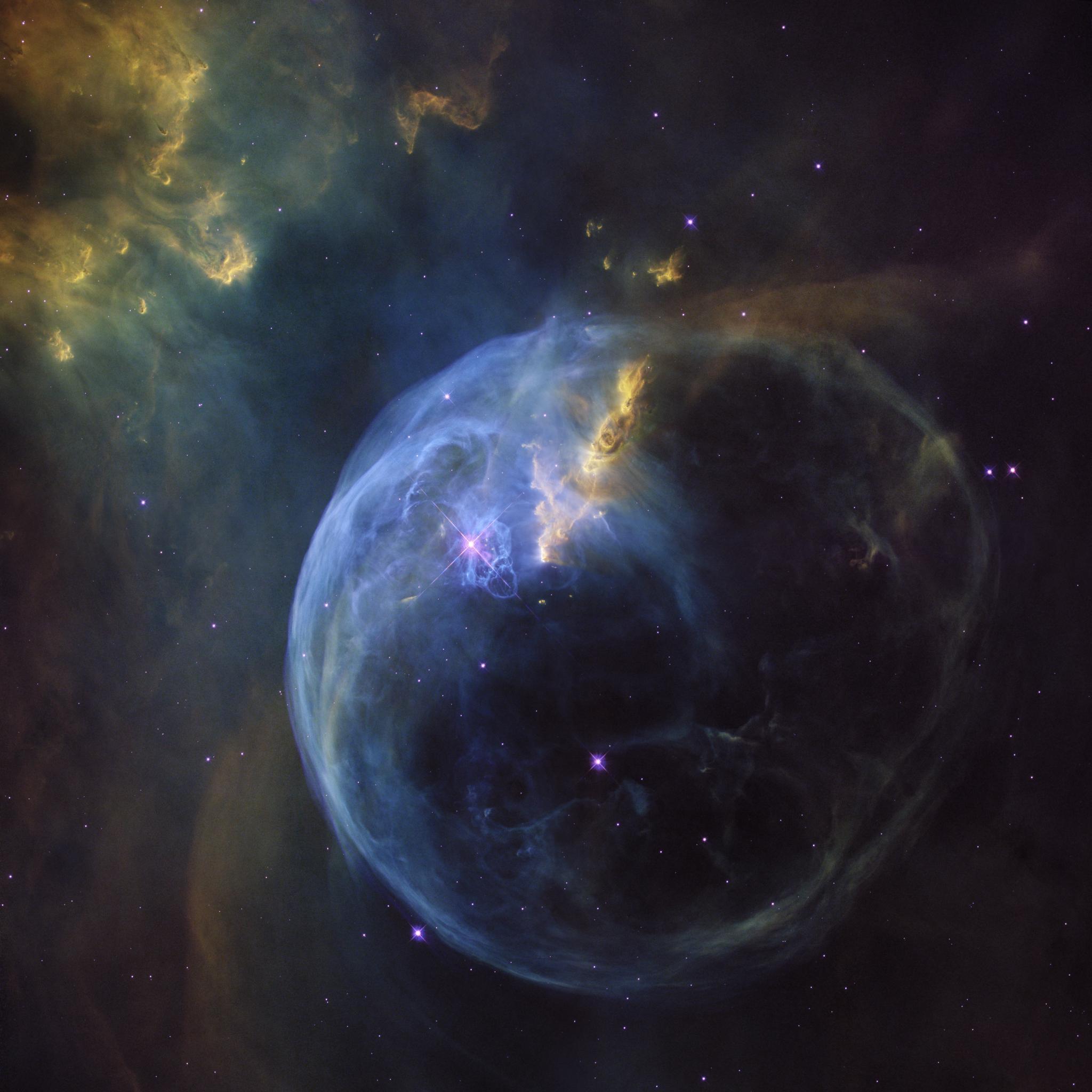}
		\end{minipage}
		\\ \vspace{3pt}
	\begin{minipage}[b]{0.49\textwidth}
		\includegraphics[width=\textwidth]{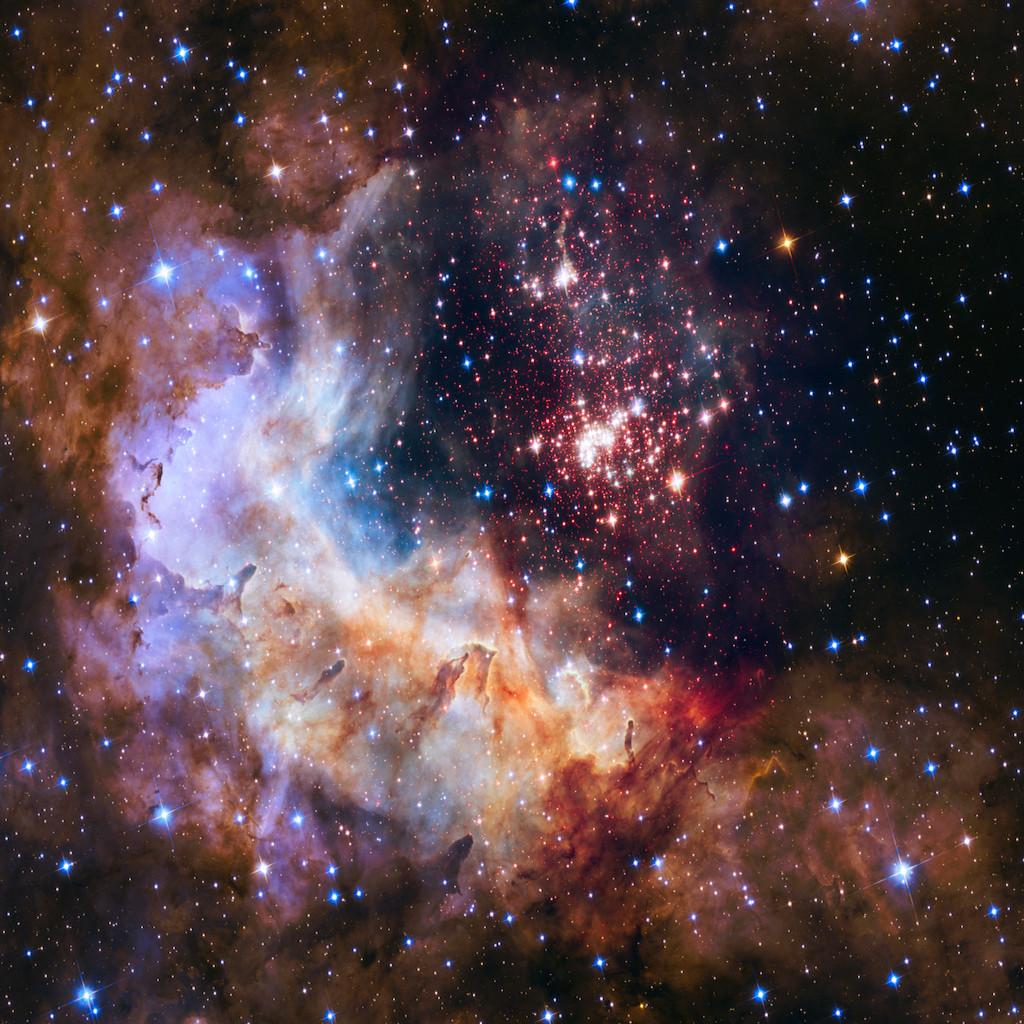}
	\end{minipage}~
	\begin{minipage}[b]{0.49\textwidth}
		\includegraphics[width=\textwidth]{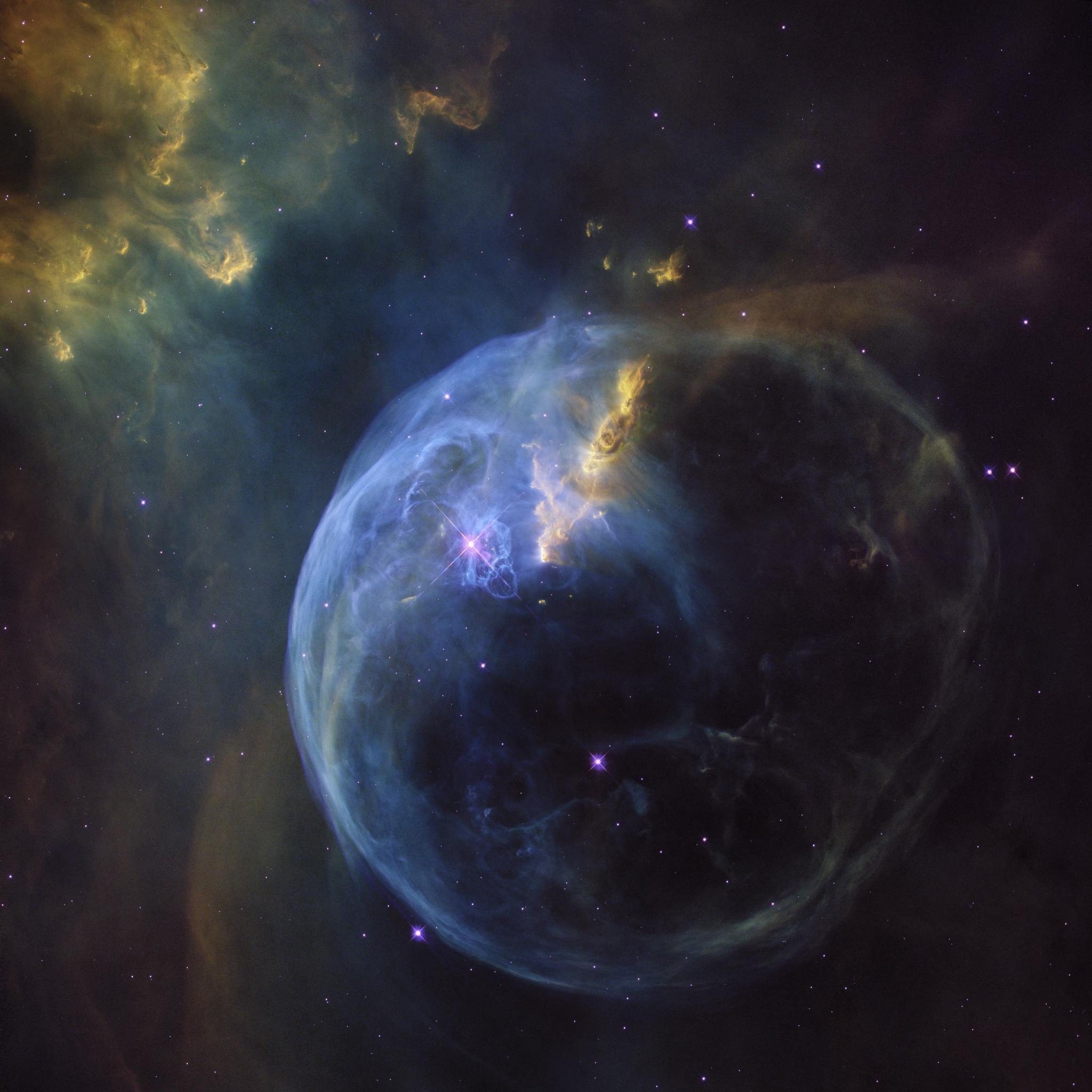}
	\end{minipage}
		\caption{Image recovery; top row are the true images, bottom row are the images produced by the subgradient method. We do not record the  images produced by the initialization as they were both completely black.
		 Dimensions of the problem: $(n,k,d,m)\approx(1024,3,2^{22},2^{24})$ (left) and $(n,k,d,m)\approx(2048,3,2^{24},2^{25})$ (right).}
		\label{fig:hubbl_rec_image}
\end{figure}

	\begin{figure}[!h]
		\centering
		\begin{minipage}[b]{0.49\textwidth}
		\includegraphics[width=\textwidth]{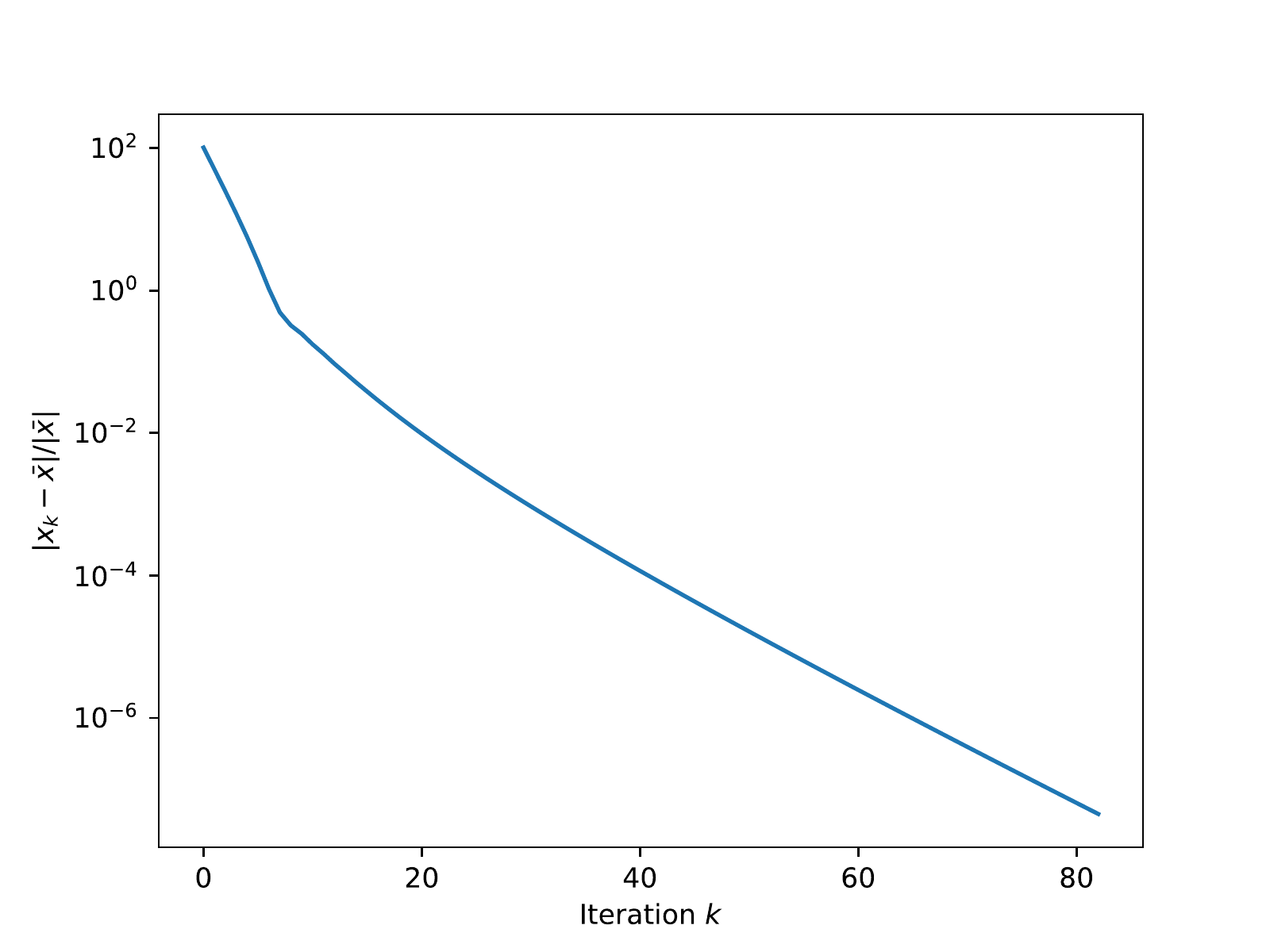}
		\end{minipage}~
				\begin{minipage}[b]{0.49\textwidth}
		\includegraphics[width=\textwidth]{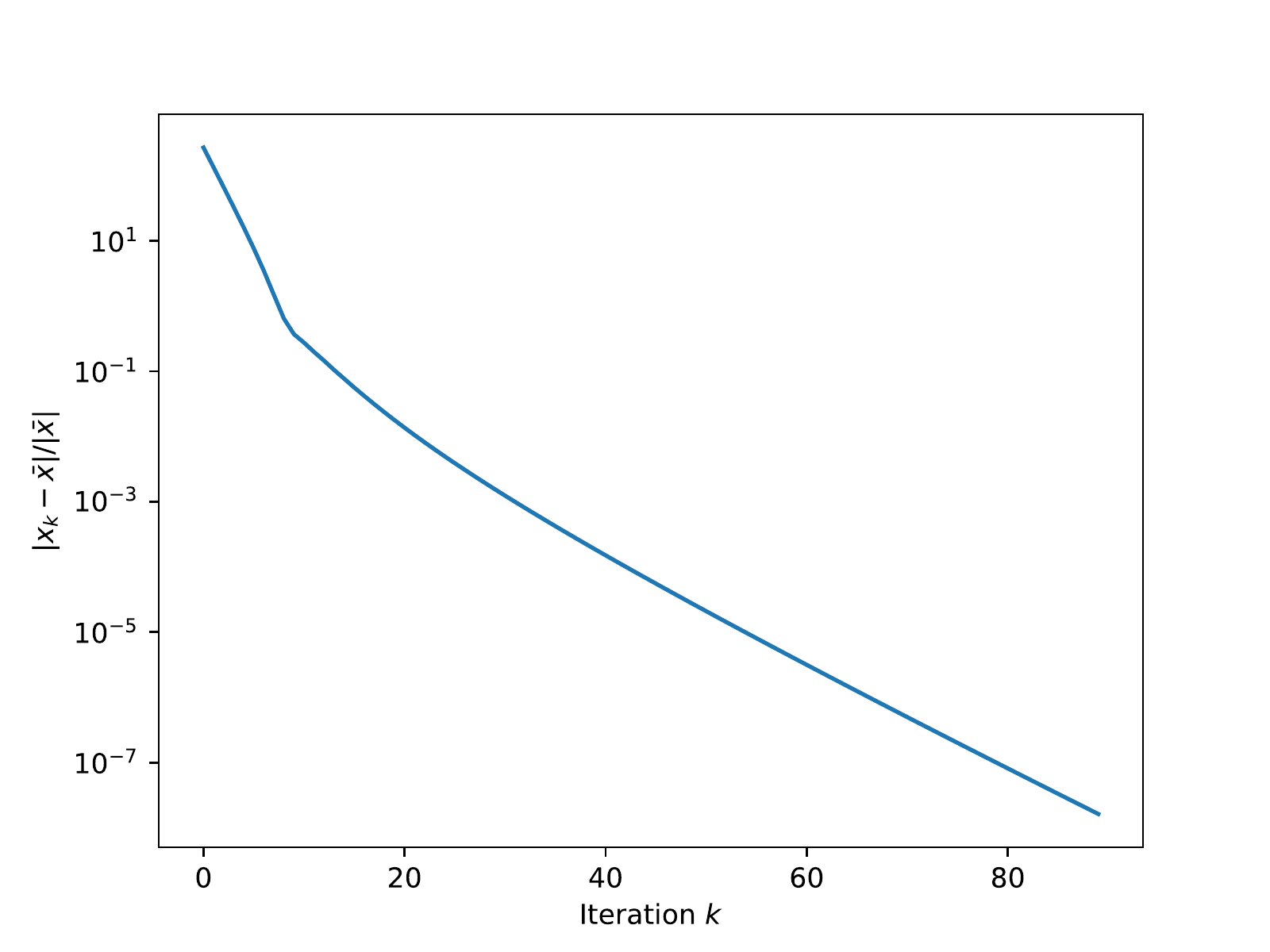}
		\end{minipage}
				\caption{Convergence plot on the two Hubble images (iterates vs. $\|x_k-\bar x\|/\|\bar x\|$).}

		\label{fig:error_hubble}
	\end{figure}

%% file: landscape.tex
		\section{Nonsmooth landscape of the robust phase retrieval}\label{sec:landscape}
		In this section, we pursue a finer analysis of the stationary points of the robust phase retrieval objective $f_S$.	To motivate the discussion, recall that under Assumptions~\ref{ass:weird_prob} and \ref{ass:subgauss}, Lemma~\ref{eq:region_for_linear} shows that there are no extraneous stationary points $x$ satisfying
		\begin{align*} 
		\min\left\{\frac{\|x - \bar x\|}{\|\bar x\|},\frac{\|x + \bar x\|}{\|\bar x\|}\right\} < \frac{\kappa_{\mathrm{st}}^\ast p_0}{4\sigma^2}.
		\end{align*}
		This result is uninformative when $x$ is far away from $\bar x$ or when $x$ is close to the origin. Therefore, it is intriguing to determine the location of {\em all} the stationary points of $f_S$. In this section, we will see that under a Gaussian observation model, the stationary points of $f_S$ cluster around the codimension two set, $\{0,\pm\bar x\}\cup(\bar x^{\perp}\cap c\cdot\mathbb{S}^{d-1})$, where $c\approx 0.4416$ is a numerical constant. 
		
		\subsection{A matrix analysis interlude}
		Before continuing, we introduce some basic matrix notation. We mostly follow  \cite{con_herm,eval,simp_var_anal}.
		 The symbol $\mathcal{S}^d$ will denote the
		Euclidean space of real symmetric $d\times d$-matrices with the trace inner product $\langle X,Y\rangle:=\trace(XY)$. 
			A function $f\colon\R^d\to{\R}$ is called {\em symmetric} if equality, $f(\sigma x)=f(x)$, holds for all coordinate permutations $\sigma$.
			For any symmetric function $f\colon\R^d\to{\R}$, we define the induced function on the symmetric matrices $f_{\lambda}\colon\mathcal{S}^d\to\R$  as the composition 
			$$f_{\lambda}(X):=f(\lambda(X)),$$
			where $\lambda\colon\mathcal{S}^d\to\R^d$ assigns to each  matrix $X\in \mathcal{S}^d$ its eigenvalues in nonincreasing order
			$$\lambda_1(X)\geq \lambda_2(X)\geq \ldots\geq \lambda_n(X).$$
			Note that $f$ coincides with the restriction of $f_{\lambda}$ to diagonal matrices,
			$f_{\lambda}(\Diag(x))=f(x)$. Any function on $\mathcal{S}^d$ that has the form $f_{\lambda}$ for some symmetric function $f$, is called {\em spectral}. Equivalently, spectral functions on $\mathcal{S}^d$ are precisely those that are invariant under conjugation by orthogonal matrices. 
		Henceforth, let  $\mathbb{O}^d$ be the set of real $d\times d$  orthogonal matrices.
		
			Recall that two matrices $X,V\in\mathcal{S}^d$ commute if and if they can be simultaneously diagonalized. When describing variational properties of convex spectral functions, a stronger notion is needed. We say that $X,V$ admit a {\em simultaneous ordered spectral decomposition} if there exists a matrix $U\in \mathbb{O}^d$ satisfying
			$$UVU^T=\Diag(\lambda(V))\qquad \textrm{and}\qquad UXU^T=\Diag(\lambda (X)).$$
			Thus the definition stipulates that $X$ and $V$  admit a simultaneous diagonalization, where the diagonals of the two diagonal matrices are simultaneously ordered.

		The following is a foundational theorem in the convex analysis of spectral functions, due to Lewis \cite{con_herm}. An extension to the nonconvex setting was proved in \cite{eval}, while a much simplified argument was recently presented in \cite{simp_var_anal}.

		\begin{thm}[Spectral convex analysis]\label{thm:char_spec_sbdiff}
			Consider a symmetric function $f\colon\R^d\to\R\cup\{+\infty\}$. Then  $f$ is convex if and only if  $f_{\lambda}$ is convex. Moreover, if $f$ is convex, then the subdifferential $\partial f_{\lambda}(X)$ consists of all matrices
			 $V\in \mathcal{S}^d$ satisfying $\lambda(V)\in \partial f(\lambda(X))$ and such that $X$ and $V$ admit a simultaneous ordered spectral decomposition.

%
		\end{thm}

		\subsection{Landscape of the population objective}

Henceforth, we fix a point $0\neq \bar x\in\R^d$ and assume that $a\in\R^d$ is a normally distributed random vector $a\sim\mathsf{N}(0,I_d)$. In this section, we will investigate the population objective of the robust phase retrieval problem:
$$f_P(x):=\EE_a\left[ |\dotp{a, x}^2 - \dotp{a, \bar x}^2|\right].$$
Our aim is to prove the following result; see Figure~\ref{fig:norm_contour} for a graphical depiction. 
\begin{thm}[Landscape of the population objective]\label{thm:pop_obj} {\hfill \\ }
The stationary points of the population objective $f_P$ are precisely 
\begin{equation}\label{eqn:stat_set}
\{0\}\cup\{\pm\bar x\}\cup \{x\in \bar x^{\perp}: \|x\|=c\cdot\|\bar x\|\},
\end{equation}
where 
$c>0$ (approx. $c\approx 0.4416$) is the unique solution of the equation
 $\frac{\pi}{4}=\frac{c}{1+c^2}+\arctan\left(c\right).$ 
\end{thm}

\begin{figure}[h!]
	\centering
	\includegraphics[scale=1.2]{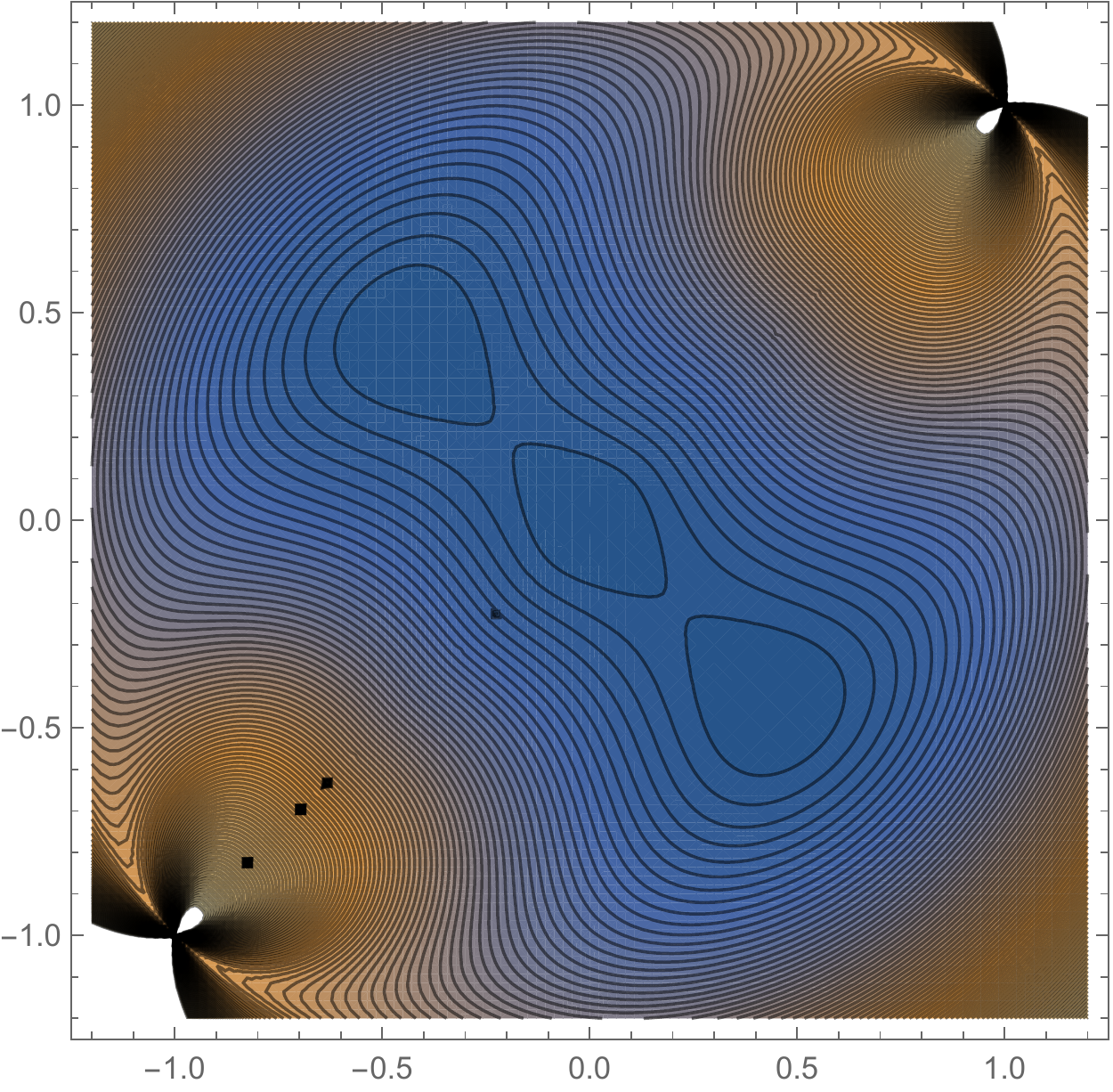}
	\caption{The contour plot of the function $x\mapsto\|\nabla f_P(x)\|$, where $\bar x=(1,1)$. The global minimizers of $f_P$ are $\pm\bar x$, while the three extraneous stationary points are $(0,0)$ and $\pm c(-1,1)$, where $c\approx 0.4416$.}
	\label{fig:norm_contour}
\end{figure}

Theorem~\ref{thm:pop_obj} provides an exact characterization of the stationary points of the population objective $f_P$. Looking ahead, when we will pass to the subsampled objective $f_S$ in Section~\ref{sec:concent_stab}, we will show that every stationary point of $f_S$ is {\em close} to an {\em approximately} stationary point of $f_P$. Therefore it will be useful to have an extension of Theorem~\ref{thm:pop_obj} that locates approximately stationary points of $f_P$. 
This is the content of the following theorem.

\begin{thm}[Location of approximate stationary points]\label{thm:pop_obj_quant_vers_main}
There exists a numerical constant $\gamma>0$ such that the following
holds. For any point $x\in\R^d$ with $$\varepsilon:=\dist(0;\partial
f_P(x))\leq \gamma \|x\|,$$  it must be the case that $\|x\| \lesssim \|\bar x\|$ and $x$ satisfies either
\begin{equation*}
\|x\|\|x-\bar x\|\|x+\bar x\|\lesssim \varepsilon \|\bar
x\|^2\qquad \textrm{or}\qquad\left\{\begin{aligned}
\left|\|x\|-c\|\bar x\| \right|&\lesssim \varepsilon \frac{\|\bar
  x\|}{\|x\|} \\
|\langle x,\bar x\rangle|&\lesssim \varepsilon \|\bar x\|
\end{aligned}\right\},
\end{equation*}
where $c>0$ is the unique solution of the equation
 $\frac{\pi}{4}=\frac{c}{1+c^2}+\arctan\left(c\right).$
\end{thm}

We present the proofs of Theorem~\ref{thm:pop_obj} in Section~\ref{sec:proof_5.2}, and defer the proof of Theorem~\ref{thm:pop_obj_quant_vers_main} to the Appendix (Section~\ref{sec:cray_sec_inexact}), as the latter requires a much more delicate argument.
At their core, the arguments rely on the observation that the population objective $f_P(x)$ depends on the input vector $x$ only through the eigenvalues of the rank two matrix $xx^T-\bar x\bar x^T$. This observation was already implicitly used by Cand{\`e}s et al. \cite{phase_lift}. Since this matrix will appear often in the arguments, we will use the symbol $X:=xx^T-\bar x\bar x^T$ throughout. For ease of reference, we record the following simple observation: the matrix $X$ is typically indefinite.

 \begin{lem}[Eigenvalues of the rank two matrix]\label{lem:rand_eigen}
 	Suppose  $x$ and $\bar x$ are not collinear. Then $X$ has exactly one strictly positive and one strictly negative eigenvalue.
 \end{lem} 
 \begin{proof}
 	Suppose the claim is false. Then either $X$ is positive semidefinite or negative semidefinite. Let us dispense with the first case.
 	Observe $X\succeq 0$ if and only if $(x^Tv)^2-(\bar x^T v)^2\geq 0$ for all $v$. Hence if $X$ were positive semidefinite, we would deduce $x^{\perp}\subset \bar x^{\perp}$; that is, $x$ and $\bar x$ are collinear, a contradiction. The case  $X\preceq 0$ is analogous.
 \end{proof}

The following lemma, as we alluded to above, shows that $f_P(x)$ depends on  $x$ only through the eigenvalues of the rank two matrix $X=xx^T-\bar x\bar x^T$.
\begin{lem}[Spectral representation of the population objective]\label{lem:spec_represe}{\hfill \\ }
For all points $x\in\R^d$, equality holds:
\begin{equation}\label{eqn:pop_spec}
f_P(x)=\EE_{v} \left[\Big| \langle \lambda(X), v\rangle\Big|\right],
\end{equation}
where $v_i\in\R$ are i.i.d. chi-squared random variables $v_i\sim\chi^2_1$. 
\end{lem}
\begin{proof}
Observe the equalities:
\begin{align*}
f_P(x) =\EE_a\left[ |\dotp{a, x}^2 - \dotp{a, \bar x}^2|\right]&= \EE_a[|\langle a,x-\bar x \rangle\langle a,x+\bar x \rangle|]\\
&=\EE_a[|(x-\bar x)^T a a^T(x+\bar x)|]\\
&=  \EE_a\left[| \trace\left(a^T (x + \bar x)(x-\bar  x)^Ta\right)|\right].
\end{align*}
Thus in terms of the matrix $M := (x + \bar x)(x-\bar x)^T$, we have  $f_P(x)=\EE_a\left[| \trace\left(a^TM a\right)|\right]$. Taking into account the equalities  $a^TM a=a^T\left(\frac{M+M^T}{2}\right)a=a^TXa$, we deduce 
$$f_P(x)=\EE_a\left[| \trace\left(a^TX a\right)|\right].$$
Form now an eigenvalue decomposition 
$
X = U \Diag(\lambda(X)) U^T$,
where $U\in \R^{d\times d}$ is an orthogonal matrix. Rotation invariance of the Gaussian distribution then implies
$$
 \EE_a\left[| \trace(a^TXa)|\right]  = \EE_a\left[| \trace((Ua)^T X(Ua))|\right] = \EE_{u} \left[\left|\sum_{i=1}^d \lambda_i(X) u_i^2\right|\right],
$$
where $u_i$ are i.i.d standard normals. The result follows.
\end{proof}

Thus Lemma~\ref{lem:spec_represe} shows that the population objective $f_P$ is a spectral function of $X$. Combined with Lemma~\ref{lem:rand_eigen}, we deduce that there are two ways to rewrite the population objective in composite form:
$$f_P(x)=\varphi_{\lambda}(X)\qquad \textrm{ and }\qquad f_P(x)=\zeta(\lambda_1(X),\lambda_d(X)),$$
where 
\begin{equation}\label{eqn:phipsi}
\varphi(z):=\EE_{v} \left[\Big|\langle z, v\rangle\Big|\right]\qquad \textrm{and}\qquad \zeta(y_1,y_2):= \EE_{v_1,v_2} \left[| v_1 y_1+ v_2 y_2|\right].
\end{equation}
Notice that $\varphi$ and $\zeta$ are norms on $\R^d$ and $\R^2$, respectively. It is instructive to compute $\zeta$ in closed form, yielding the following lemma. Since the proof is a straightforward computation, we have placed it in the appendix.

\begin{lem}[Explicit representation of the outer function]\label{lem:cray_compute} {\hfill \\ } Let $v_1, v_2\sim\chi^2_1$ be i.i.d. chi-squared. Then for all real $(y_1,y_2)\in \R_+\times\R_{-}$, equality holds:
$$
\EE_{v_1,v_2} \left[| v_1 y_1  + v_2 y_2 |\right] = 
						\frac{4}{\pi}
                                                        \left [
                                                        (y_1+y_2)
                                                        \arctan \left
                                                        (
                                                        \sqrt{-\frac{y_1}{y_2}
                                                        }\right ) +
                                                        \sqrt{-y_1
                                                        y_2}
                                                        \right ] -
                                                        (y_1+y_2).                                                           
$$


\end{lem}

\begin{figure}[h!]
	\centering
		\includegraphics[scale=0.7]{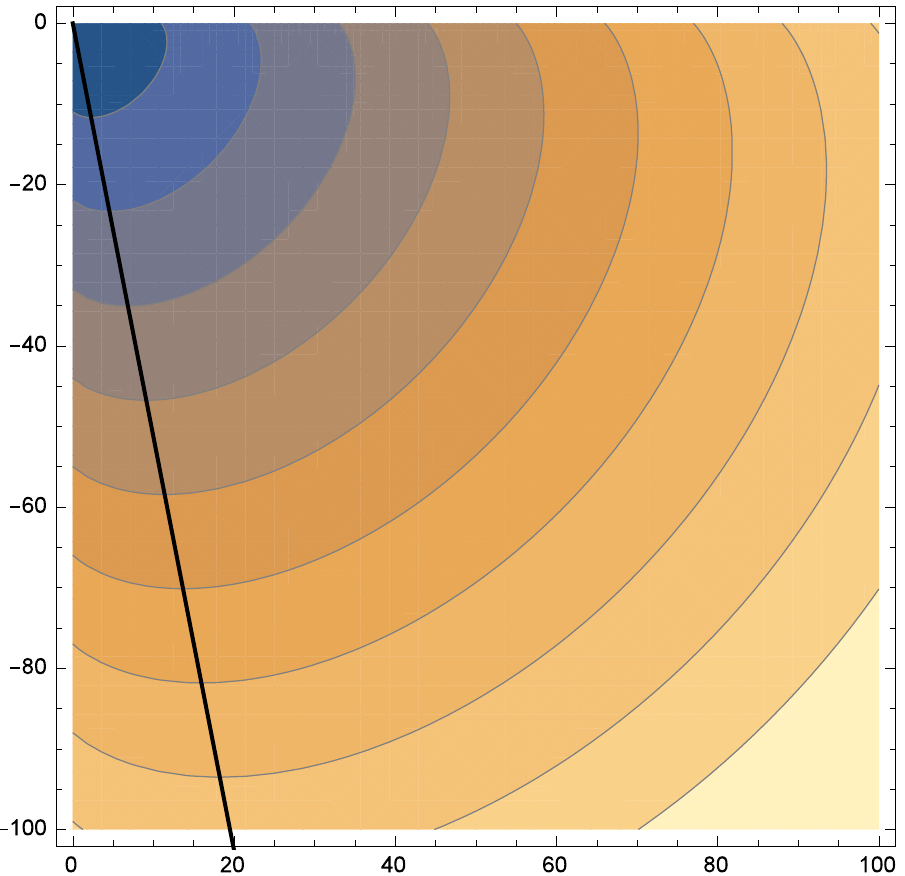}
		\caption{Contour plot of the function $\zeta(y_1,y_2):=\EE_{v_1,v_2} \left[| v_1 y_1  + v_2 y_2 |\right]$ on $\R_+\times\R_{-}$. The black line depicts all points $(y_1,y_2)$ with $\nabla_{y_1}\zeta(y_1,y_2)=0$; for the explanation of the significance of this line, see Lemma~\ref{lem:weird_comput_slope}.}
                        \label{fig:contours_outer}                            	
                                                    	
\end{figure}

Thus we have arrived at the following explicit representation of $f_P(x)$.  Figure~\ref{fig:pop_obj} in the introduction depicts the graph and the contours of the population objective.

\begin{cor}[Explicit representation of the population objective]{\hfill \\ }
The explicit representation holds:
$$f_{P}(x)=\frac{4}{\pi}\left [
\trace(X)\cdot
\arctan \left
(
\sqrt{\left|\frac{\lambda_{\max}(X)}{\lambda_{\min}(X)}\right|
}\right ) +
\sqrt{|\lambda_{\max}(X)
	\lambda_{\min}(X)|}
\right ] -
\trace(X).$$
\end{cor}



\input{calculations}

%% file: calculations.tex


\subsection{Proof of Theorem~\ref{thm:pop_obj}}\label{sec:proof_5.2}
We next move on to the proof of Theorem~\ref{thm:pop_obj}. Let us first dispense with the easy implication, namely that every point in the set 
\eqref{eqn:stat_set} is indeed stationary for $f_P$; in the process, we will see how the slope $c\approx 0.4416$ arises. Clearly $\pm \bar x$ are minimizers of $f_P$ and are therefore stationary. The chain rule $\partial f_P(x)=\partial \varphi_{\lambda}(X)x$ implies that $x=0$ is stationary as well.
Fix now a point $x\in \bar x^{\perp}\setminus\{0\}$. Observe that the extremal eigenvalues of $X$ are
$$\lambda_1(X)=\|x\|^2\qquad \textrm{and}\qquad \lambda_d(X)=-\|\bar x\|^ 2,$$
with corresponding eigenvectors 
$$e_1:=\frac{x}{\| x\|}\qquad \textrm{and}\qquad e_d:=\frac{\bar x}{\|\bar x\|}.$$
Since $\lambda_1(X)$ and $\lambda_d(X)$ each have multiplicity one, the individual eigenvalue functions $\lambda_1(\cdot)$ and $\lambda_d(\cdot)$ are smooth at $X$ with gradients 
$$\nabla \lambda_1(X)=e_1e_1^T\qquad \textrm{and}\qquad \nabla \lambda_d(X)=e_de_d^T.$$ See for example \cite[Theorem 5.11]{baby_kato}. 
Setting $(y_1,y_2):=(\|x\|^2,-\|\bar x\|^2)$ and applying the chain rule to the decomposition $f_P(x)=\zeta(\lambda_1(X),\lambda_d(X))$ shows
\begin{align*}
\nabla f_P(x)&=\left(\nabla_{y_1}\zeta(y_1,y_2)e_1e_1^T+\nabla_{y_2}\zeta(y_1,y_2)e_de_d^T\right)x=\nabla_{y_1}\zeta(y_1,y_2)x.
\end{align*}
Thus a point $x\in \bar x^{\perp}\setminus\{0\}$ is stationary for $f_P$ if and only if the partial derivative $\nabla_{y_1}\zeta(y_1,y_2)$ vanishes. The points $(y_1,y_2)$ satisfying the equation $0=\nabla_{y_1}\zeta(y_1,y_2)$ trace out exactly the 
line depicted in Figure~\ref{fig:contours_outer}.

	\begin{lem}\label{lem:weird_comput_slope}
		The solutions of the equation $0=\nabla_{y_1} \zeta(y_1,y_2)$  on $\R_{++}\times \R_{--}$ are precisely the tuples $\{(c^2y,-y)\}_{y>0}$, where $c>0$ is the unique solution of the equation
		$$\frac{\pi}{4}=\frac{c}{1+c^2}+\arctan\left(c\right).$$
		Note $c\approx 0.4416$.
	\end{lem}
	\begin{proof}
		Differentiating shows that $\omega(c):=\frac{c}{1+c^2}+\arctan\left(c\right)$ is a continuous strictly increasing function on $[0,+\infty)$ with $\omega(0)=0$ and $\lim_{c\to+\infty}\omega(c)=\pi/2$. Hence the equation $\pi/4=\omega(c)$ has a unique solution in the set $(0,\infty)$. 
		A short computation yields the expression 
		$$\nabla_{y_1} \zeta(y_1,y_2)=\frac{4}{\pi}\left(\frac{y_1+y_2}{2\sqrt{-y_1/y_2}(y_1-y_2)}-\frac{y_2}{2\sqrt{-y_1y_2}}+\arctan\left(\sqrt{-\frac{y_1}{y_2}}\right)\right)-1.$$
		Set $y_1=-c^2y_2$ for some $c>0$ and $y_2<0$. Then plugging in this value of $y_1$,  equality $0=\nabla_{y_1} \zeta(y_1,y_2)$ holds  if and only if
		$$\pi/4=\left(\frac{c}{1+c^2}+\arctan\left(c\right)\right).$$
		This equation is independent of $y_1$ and its solution in $c$ is exactly the value satisfying $\pi/4=\omega(c)$.
	\end{proof}

%

Thus we have proved the following. 
\begin{proposition}\label{prop:one_dir}
Let $c>0$  be the unique solution of the equation
 $\frac{\pi}{4}=\frac{c}{1+c^2}+\arctan\left(c\right)$. 
Then a point $x\in \bar x^{\perp}\setminus\{0\}$ is stationary for $f_P$ if and only if equality $\| x\|=c\|\bar x\|$ holds.
\end{proposition}

In particular, we have proved one implication in Theorem~\ref{thm:pop_obj}.
 To prove the converse, we must show that every stationary point of $f_P$ lies in the set \eqref{eqn:stat_set}. Various approaches are possible based  either on the decomposition $f_P(x)=\varphi_{\lambda}(X)$ or $f_P(x)=\zeta(\lambda_1(X),\lambda_d(X)).$ We will focus on the former. We will prove a strong result about the location of stationary points of arbitrary convex spectral functions of $X$. 
 Indeed, it will be more convenient to consider the more abstract setting as follows.

%

 
Throughout, we fix a symmetric convex  function $f\colon\R^d\to\R$ and a point $0\neq \bar x\in\R^d$, and define the function 
$$g(x):=f_\lambda(xx^T-\bar x\bar x^T).$$ 
Note, the population objective $f_P$ has this representation with $f=\varphi$.
The chain rule directly implies $$\partial g(x)=\partial f_{\lambda}(X)x.$$ 
Therefore, using Theorem~\ref{thm:char_spec_sbdiff} let us also fix a matrix $V\in\partial f_{\lambda}(X)$ and 
a matrix $U\in \mathbb{O}^d$ satisfying
$$\lambda(V)\in \partial f(\lambda(X)),\qquad V=U(\Diag (\lambda(V))U^T,\qquad \textrm{ and }\qquad X = U \Diag(\lambda(X)) U^T.$$

The following  two elementary lemmas will form the core of the argument.

\begin{lem}[Eigenvalue correlation]\label{lem:correlation} {\hfill \\ }  
	The following are true. 
\begin{enumerate}
\item\label{claim:1_corr} \textbf{Eigenvalues.} We have $\lambda_i(X) = \dotp{U_i, x}^2 - \dotp{U_i, \bar x}^2$ for $i \in \{1, d\}$, and consequently  
\begin{equation}
\begin{aligned} \label{eq: eigenvalue_bound}
0 \le \lambda_1(X) &\leq \dotp{U_1, x}^2 \leq \|x\|^2 \\
0 \le -\lambda_d(X) &\leq  \dotp{U_d, \bar x}^2 \leq  \|\bar x\|^2.
\end{aligned}
\end{equation}
\item\label{claim:2_corr} \textbf{Anticorrelation.} Equality holds:
\begin{align*}
 \dotp{U_1, x}\dotp{U_d, x} =  \dotp{U_1, \bar x} \dotp{U_d, \bar x}.
\end{align*}
\item\label{claim:3_corr} \textbf{Correlation.} Provided $x\notin \{\pm\bar x\}$, we have $\spann\{x,\bar x\}\subset\spann\{U_1,U_d\}$ and 
$$
\dotp{x, \bar x} = \dotp{U_1, x}\dotp{U_1, \bar x} + \dotp{ U_d, x}\dotp{U_d, \bar x}.
$$
\end{enumerate}
\end{lem}
\begin{proof} 
From the eigenvalue decomposition, we obtain
\begin{align*}
\lambda_1(X) &= U_1^T X U_1 =  \dotp{U_1, x}^2 - 
  \dotp{U_1, \bar x}^2  \\
\lambda_d(X) &= U_d^T X U_d =  \dotp{U_d, x}^2 - 
                         \dotp{U_d, \bar x}^2 .
\end{align*}
Taking into account that always $\lambda_1(X)\geq 0$ and $\lambda_1(X)\leq 0$ (Lemma~\ref{lem:rand_eigen}), we conclude $\lambda_1(X) \le
\dotp{U_1,x}^2$ and $\lambda_d(X) \ge
-\dotp{U_d,\bar x}^2$. Claim~\ref{claim:1_corr} follows.
For Claim \ref{claim:2_corr}, simply observe
\begin{align*}
0 = U_d^T X U_1 = \dotp{U_1,x} \dotp{U_d, x}
  - \dotp{U_1, \bar x} \dotp{U_d, \bar x}. 
\end{align*}
To see Claim \ref{claim:3_corr}, for each $i\in \{1,d\}$  notice 
$$\langle U_i,x\rangle x -\langle U_i,\bar x\rangle \bar x=XU_i=\lambda_i(X)U_i.$$
Suppose $x\notin \{\pm\bar x\}$. Then if $x$  and $\bar x$ are not collinear, we may divide through by $\lambda_i(X)$ and deduce, $\spann\{U_1,U_d\}=\spann\{x,\bar x\}$. On the other hand, if $x$ and $\bar x$ are collinear, then exactly one $\lambda_1$ or $\lambda_d$ is nonzero, and then $x$ lies in the span of the corresponding column of $U$. In either case, we may write $x = \dotp{U_1, x} U_1 + \dotp{U_d, x} U_d$
and $\bar x = \dotp{U_1, \bar x} U_1 + \dotp{U_d, \bar x} U_d$ in terms
of their orthogonal expansions. We deduce 
$$
\dotp{x, \bar x} = \dotp{\dotp{U_1, x} U_1 + \dotp{U_d, x} U_d, \dotp{U_1, \bar x} U_1 + \dotp{U_d, \bar x} U_d} = \dotp{U_1, x}\dotp{U_1, \bar x} + \dotp{ U_d, x}\dotp{U_d, \bar x},
$$
as claimed.
\end{proof}


\begin{lem}[Spectral subdifferential]\label{lem:lower_bound_derivative}
	The following hold:
\begin{align}\label{eq:gradient_lower_bound}
    &\max\left\{|\lambda_1(V)\dotp{U_1, x}|, |\lambda_d(V)\dotp{U_d, x}|\right\} \leq \|Vx\|, 
\end{align}
and
\begin{align}\label{eq:correlation_inequality}
   g(x) -g(\bar x) \leq \lambda_1(V) \lambda_1(X) + \lambda_d(V) \lambda_d(X).
\end{align}
\end{lem}
\begin{proof}
	
		To see~\eqref{eq:gradient_lower_bound}, observe that for all unit vectors $z \in \mathbb{S}^{d-1}$, we have $\|Vx\| \geq \dotp{z, Vx}$. Thus, testing against all $z \in \{\pm U_1, \pm U_d\}$ yields the lower bounds \eqref{eq:gradient_lower_bound}. 
		To prove the final bound \eqref{eq:correlation_inequality}, we exploit the convexity of $f_\lambda$. The subgradient inequality implies
\begin{align*}  f_\lambda (X) - f_\lambda (0) &\leq \dotp{V, X} 
= \lambda_1(V) \lambda_1(X) + \lambda_d(V) \lambda_d(X).
\end{align*}
The result follows.
\end{proof}

The following corollary follows quickly from the previous two lemmas. 

\begin{cor}[Stationary point inclusion]\label{cor:find_stat}
Suppose
that $x$ is stationary for $g$, that is $Vx=0$. Then one of the following conditions holds:
\begin{enumerate}
\item\label{it:1_cor512} $g(x) \leq g(\bar x)$
\item\label{it:2_cor512} $x = 0$
\item\label{it:3_cor512} $\dotp{x , \bar x} = 0$, $\lambda_1(V) = 0$. 
\end{enumerate}
Moreover, if $\bar x$ minimizes $g$, then a point $x$ is stationary for $g$ if and only if $x$ satisfies \ref{it:1_cor512}, \ref{it:2_cor512}, or \ref{it:3_cor512}.
\end{cor}
\begin{proof}
Suppose $Vx=0$ and that the first two conditions fail, that is $x\neq 0$ and $g(x)> g(\bar x)$. We will show that the third condition holds. To this end, inequalities \eqref{eq:gradient_lower_bound} and  \eqref{eq:correlation_inequality}, along with Lemma~\ref{lem:correlation}, directly imply the following:
\begin{align}
0 < g(x)-g(\bar x) &\leq \lambda_1(V) \lambda_1(X) + \lambda_d(V) \lambda_d(X),\label{eqn:proof_ineq_corr}
\\
\lambda_1(V)\dotp{U_1, x} &= \lambda_d(V)\dotp{U_d, x} = 0,\label{eqn:proof_alt} \\
x&=\dotp{U_1,x}U_1+\dotp{U_d,x}U_d.\label{eqn:proof_alt3}
\end{align}
Aiming towards a contradiction, suppose  $\lambda_1(V) \neq 0$. Then  \eqref{eqn:proof_alt} and \eqref{eqn:proof_alt3} imply $\dotp{U_1, x} = 0$ and $\dotp{U_d, x} \neq 0$.
The second equation in \eqref{eqn:proof_alt}, in turn, yields $\lambda_d(V)=0$.
 Appealing to  Lemma~\ref{lem:correlation}, we moreover deduce  
$$0\leq \lambda_1(X) = \dotp{U_1, x}^2 - \dotp{U_1, \bar x}^2 \leq 0.$$ Thus  $\lambda_1(X) = 0$ and therefore the right-hand-side of  \eqref{eqn:proof_ineq_corr} is zero, a contradiction.
We have shown the equality $\lambda_1(V) = 0$, as claimed. 

Inequality \eqref{eqn:proof_ineq_corr} implies   $\lambda_d(V) \neq 0$  and $\lambda_d(X)\neq 0$, and hence by Inequality~\eqref{eqn:proof_alt}, we have $\dotp{U_d, x} = 0$. Combining the latter equality with Lemma~\ref{lem:correlation}, we conclude $0 = \dotp{U_1, x}\dotp{U_d, x} = \dotp{U_1, \bar x} \dotp{U_d, \bar x}$. Note 
$ \dotp{U_d, \bar x} \neq 0$, since otherwise we would get $\lambda_d(X)=0$ by \eqref{eq: eigenvalue_bound}. 
We conclude $\dotp{U_1, \bar x}  = 0$.
Finally,  Lemma~\ref{lem:correlation} then yields 
$$
\dotp{x, \bar x} = \dotp{U_1, x}\dotp{U_1, \bar x} + \dotp{ U_d, x}\dotp{U_d, \bar x} = 0,
$$ 
thereby completing the proof.

Now suppose that $\bar x$ minimizes $g$. Clearly $\pm\bar x$ is a stationary point of $g$. In addition, $0$ is a stationary point of $g$ because $V \cdot 0 = 0$. Thus, it remains to show that all points satisfying \ref{it:3_cor512} are stationary. Thus suppose $x$ satisfies $\ref{it:3_cor512}$ and  $x\neq 0$.
Then the eigenvalues of $X$ are precisely $\|x\|^2$ and $-\|\bar x\|^2$ with eigenvectors $U_1=\pm\frac{x}{\|x\|}$ and $U_d=\pm\frac{\bar x}{\|\bar x\|}$, respectively.
Thus, we have $U^T Vx = \Diag(\lambda(V)) U^T x = (\lambda_1(V) \dotp{U_1, x}, 0,\ldots,0, \lambda_d(V) \dotp{U_d, x})^T = 0.$ We conclude $Vx=0$, as required. 
\end{proof}

The proof of Theorem~\ref{thm:pop_obj} is now immediate.
\begin{proof}[Proof of Theorem~\ref{thm:pop_obj}]
We have already proved that every point in the set \eqref{eqn:stat_set} is stationary for $f_P$ (Proposition~\ref{prop:one_dir}). Thus we focus on the converse. In light of Proposition~\ref{prop:one_dir}, it is sufficient to show that every stationary point $x$ of $f_P$ lies in the set $\{0,\pm \bar x\}\cup x^{\perp}$. This is immediate from Corollary~\ref{cor:find_stat} under the identification $f_P(x)=g(x)=\varphi_{\lambda}(X)$.	
\end{proof}

%% file: concentration.tex
\section{Concentration and stability}\label{sec:concent_stab}
Having determined the stationary points of the population objective $f_P$, we next turn to the stationary points of $f_S$. Our strategy is to show that with high probability, every stationary point of $f_S$ is close to some stationary point of $f_P$. The difficulty is that it is not true that $\partial f_S(x)$ concentrates around $\partial f_P(x)$. Instead, we will see that the graphs of the two subdifferentials $\partial f_S$ and $\partial f_P$ concentrate, which is sufficient for our purposes. 
Our argument will rely on two basic properties, namely (1) the subsampled objective $ f_S$  concentrates well around  $f_P$, and (2) the function $f_S$ is  weakly convex.

\subsection{Concentration of subdifferential graphs}
Armed with the concentration (Theorem~\ref{thm:em}) and the weak convexity (Theorem~\ref{lem:weak_conv}) guarantees, we can show that the graphs of $\partial f_P$ and $\partial f_S$ are close. The following theorem will be our main technical tool, and is of interest in its own right. In essence, the result is a quantitative extension of the celebrated Attouch's convergence theorem \cite{attouch} in convex analysis. Henceforth, for any function $l\colon\R^d\to\overline{\R}$ and a point $\bar x\in\R^d$, with $f(\bar x)$ finite, we define the local Lipschitz constant $$\lip(l;\bar x):=\ls_{x\to\bar x}\frac{|l(x)-l(\bar x)|}{\|x-\bar x\|}.$$



\begin{thm}[Comparison]\label{thm:compar}
	Consider four lsc functions $f,g,l,u\colon \R^d\to\overline{\R}$ and a pair $(x,v)\in \gph\partial g$. 
Suppose that $l$ is locally Lipschitz continuous and that the  following conditions	
	$$
	\left\{\begin{aligned}
	l(y)\leq & f(y)-g(y)\leq u(y)\\
	g(y)\geq g(& x)+\langle  v,y- x\rangle-\frac{\rho}{2}\|y- x\|^2
	\end{aligned}
	\right\}\qquad \textrm{hold for all points }y\in \R^d.$$
	Then for any $\gamma>0$, there exists a point $\hat x$ satisfying 
	$$\|\hat x- x\|\leq 2\gamma\qquad \textrm{ and }\qquad \dist( v;\partial f(\hat x))\leq 2\rho\gamma+ \frac{u( x)-l( x)}{\gamma}+\lip(l;\hat x).$$	
	In particular, if $l(\cdot)$ is constant, we have the estimate
	\begin{equation}\label{eqn:subdiff_close}
	\dist\Big(( x, v),\gph\partial f\Big)\leq \sqrt{4(\rho+\sqrt{2+\rho^2})}\cdot \sqrt{u({x})-l( x)}.
	\end{equation}
	\end{thm}
\begin{proof}
	From the two assumptions, for any point $y\in \R^d$ we have  
	$$f(y)\geq g(y)+l(y)\geq g( x)+l(y)+\langle  v,y- x\rangle-\frac{\rho}{2}\|y- x\|^2.$$
	Define the function 
	$$\zeta(y):=f(y)-\langle  v,y- x\rangle+\frac{\rho}{2}\|y- x\|^2 - l(y).$$
	Clearly then we have
	\begin{equation}\label{eqn:upper_bound_error}
	\zeta( x)-\inf \zeta\leq f( x) - l( x) -g( x)\leq u( x)-l( x).
	\end{equation}
	Choose now any minimizer
	$$\hat x\in\argmin_y~ \left\{\zeta(y) +\frac{u( x)-l( x)}{4\gamma^2}\cdot \|y- x\|^2\right\}.$$
	First order optimality conditions and the sum rule \cite[Exercise 10.10]{RW98} immediately imply 
	\begin{equation*}
	\frac{u( x)-l( x)}{2\gamma^2}\cdot (x-\hat x)\in \partial \zeta(\hat x)\subset \partial f(\hat x)-v+\rho(\hat x-x)+\lip(l;\hat x)B(0,1),
	\end{equation*}
	and hence
	\begin{equation}\label{eqn:opt_cond2}
	\dist( v;\partial f(\hat x))\leq \frac{u( x)-l( x)}{2\gamma^2}\cdot\|\hat x- x\|+\rho\|\hat x- x\|+\lip(l;\hat x).
	\end{equation}
Next, we estimate the distance $\|\hat x- x\|$. To this end, observe from the definition of $\hat x$, we have
	$$\zeta(\hat x) +\frac{u( x)-l( x)}{4\gamma^2}\cdot \|\hat x- x\|^2\leq \zeta( x)$$
	and hence
	\begin{equation}\label{eqn:dist_est}
	\frac{u( x)-l( x)}{4\gamma^2}\cdot \|\hat x- x\|^2\leq \zeta( x)-\zeta(\hat x)\leq u( x)-l( x),
	\end{equation}
	where the last inequality follows from \eqref{eqn:upper_bound_error}.
	In the case $u(x)=l(x)$, we deduce $\zeta(x)=
\zeta(\hat x)$. Thus we equally well could have set $\hat x=x$, and the theorem follows immediately from $\eqref{eqn:opt_cond2}$.	
	On the other hand, in the setting  $u(x)>l(x)$, the inequality \eqref{eqn:dist_est} immediately yields $\|\hat x- x\|\leq 2\gamma$, as claimed. Combining this inequality with \eqref{eqn:opt_cond2} then gives the desired guarantee
	$$\dist( v;\partial f(\hat x))\leq2 \rho\gamma+ \frac{u( x)-l( x)}{\gamma}+\lip(l;\hat x).$$	
%

	Supposing $l$ is a constant, we  have the estimate
	$$\dist\Big(( x, v),\gph\partial f\Big)\leq \sqrt{4\gamma^2+\left(2\rho\gamma+ \frac{u( x)-l( x)}{\gamma}\right)^2}.$$
	Minimizing the right-hand-side in $\gamma$ yields the choice $\gamma=\tfrac{\sqrt{u( x)-l( x)}}{(8+4\rho^2)^{1/4}}$. With this value of $\gamma$, a quick computation yields the claimed guarantee \eqref{eqn:subdiff_close}.
\end{proof}
	
	Let us now specialize the theorem to the 
	setting where the lower and upper bounds $l(\cdot), u(\cdot)$ are functions of the product $\|x-\bar x\|\cdot\|x+\bar x\|$, as in phase retrieval.
	
	\begin{cor}\label{cor:funky_right-hand-side}
	Fix two functions $f,g\colon\R^d\to\R$. Suppose that $g$ is $\rho$-weakly convex and that there is a point $\bar x$ and a real $\delta>0$ such that the inequality 
	$$|f(x)-g(x)|\leq \delta\|x-\bar x\|\cdot\|x+\bar x\|\qquad \textrm{holds for all }x\in \R^d.$$
	Then for any stationary point $x$ of $g$, there exists a point $\hat x$ satisfying 
	\begin{equation*}
	\left\{\begin{aligned}
	\|x-\hat x\|&\leq \sqrt{\tfrac{4\delta}{\rho+2\delta}}\cdot\sqrt{\|x-\bar x\|\|x+\bar x\|},\\\dist(0;\partial f(\hat x))&\leq (\delta+2\sqrt{\delta(\rho+2\delta)})\cdot(\|x-\bar x\|+\| x+\bar x\|)
	\end{aligned}\right\}.
	\end{equation*}
	
	\end{cor}
	\begin{proof}
	Set $u(x):=\delta\|x-\bar x\|\cdot\|x-\bar x\|$ and $l(x):=-\delta\|x-\bar x\|\cdot\|x-\bar x\|$ and observe $\lip(l;x)\leq \delta(\|x-\bar x\|+\|x+\bar x\|)$. Applying Theorem~\ref{thm:compar}, we deduce that for any $\gamma>0$, there exists a point $\hat x$ satisfying
		$$\|\hat x- x\|\leq 2\gamma\qquad \textrm{ and }\qquad \dist(0;\partial f(\hat x))\leq 2\rho\gamma+ \frac{2\delta\|x-\bar x\|\|x+\bar x\|}{\gamma}+\delta(\|\hat x-\bar x\|+\|\hat x+\bar x\|).$$
The triangle inequality implies
$$\|\hat x-\bar x\|\leq 2\gamma +\|x-\bar x\|\qquad \textrm{and}\qquad \|\hat x+\bar x\|\leq 2\gamma +\|x+\bar x\|,$$
and therefore 		$$\dist(0;\partial f(\hat x))\leq 2(\rho+2\delta)\gamma+ \frac{2\delta\|x-\bar x\|\|x+\bar x\|}{\gamma}+\delta(\|x-\bar x\|+\| x+\bar x\|)$$
Minimizing this expression in $\gamma>0$ yields the choice $\gamma:=\sqrt{\frac{\delta\|x-\bar x\|\|x+\bar x\|}{\rho+2\delta}}$. Plugging in this value of $\gamma$ and applying the AM-GM inequality then implies
\begin{align*}
\dist(0;\partial f(\hat x))&\leq 4\sqrt{\delta(\rho+2\delta)\|x-\bar x\|\|x+\bar x\|}+\delta(\|x-\bar x\|+\| x+\bar x\|)\\
&\leq (\delta+2\sqrt{\delta(\rho+2\delta)})(\|x-\bar x\|+\| x+\bar x\|).
\end{align*}
The result follows.
	\end{proof}

	We now arrive at the main result of the section. 

	\begin{cor}[Subsampled stationary points] \label{cor: subsample_stuff}
	Consider the robust phase retrieval objective $f_S(\cdot)$ generated from i.i.d standard Gaussian vectors. There exist numerical constants 
	$c_1, c_2>0$ such that whenever $ m\geq c_1d$, then with probability at least $1-2\exp(-\min\{m/c_1,c_2m,d^2\})$, every stationary point $x$ of $f_S$ satisfies $\|x\| \lesssim \|\bar x\|$ and one of the two conditions:
%
%
%
%
\begin{equation*}
\frac{\| x\|\| x-\bar x\|\| x+\bar x\|}{\|\bar x\|^3}\lesssim \sqrt[4]{\frac{d}{m}}
\qquad \textrm{or} \qquad \left\{\begin{aligned}
\left|\frac{\| x\|}{\|\bar x\|}-c \right|&\lesssim \sqrt[4]{\frac{d}{m}}\cdot \left( 1+\frac{\|\bar x\|}{\| x\|}\right)\\
\frac{|\langle  x,\bar x\rangle|}{\|x\| \|\bar x\|}&\lesssim \sqrt[4]{\frac{d}{m}} \cdot\frac{\|\bar x\|}{\|x\|}
\end{aligned}\right\},
\end{equation*}	
	where $c>0$ is the unique solution of the equation
 $\frac{\pi}{4}=\frac{c}{1+c^2}+\arctan\left(c\right).$
 	\end{cor}
\begin{proof}
Theorem~\ref{thm:em} shows that there exist constants $c_1, c_2>0$ such  with probability at least 
$
1 - 2 \exp(-c_1 \min \{ m, d^2\})$, we have  \begin{align}\label{eq:concentration}
	 \left| f_S(x) - f_P(x)\right| \leq \frac{c_2}{2}\left( \sqrt{\frac{d}{m}} + \frac{d}{m}\right)\|x - \bar x\| \|x + \bar x \| \qquad \textrm{ for all } x \in \RR^d.
\end{align}
Lemma~\ref{lem:weak_conv}, in turn, shows that there exist numerical constants $c_3,\rho>0$ such that provided $m\geq c_3d$, the function $f_S$ is $\rho$-weakly convex, with probability at least $1-\exp\left(-\frac{m}{c_3}\right)$.  Let us now try to apply Corollary~\ref{cor:funky_right-hand-side}. To simplify notation, define $\Delta := \sqrt{ \frac{d}{m}}$ and set 
$\delta:=c_2\Delta$. Notice $\delta\geq \frac{c_2}{2}(\Delta+\Delta^2)$ and hence we may apply Corollary~\ref{cor:funky_right-hand-side}. We deduce that with high probability, for any stationary point $x$ of $f_S$ there exists a point $\hat x\in \R^d$ satisfying
	\begin{equation}\label{eqn:brack}
	\left\{\begin{aligned}
	\|x-\hat x\|&\leq \sqrt{\tfrac{4c_2\Delta}{\rho+2c_2\Delta}}\cdot\sqrt{\|x-\bar x\|\|x+\bar x\|},\\ \dist(0;\partial f_P(\hat x))&\leq (c_2\Delta+2\sqrt{c_2\Delta(\rho+2c_2\Delta)})\cdot(\|x-\bar x\|+\| x+\bar x\|)
	\end{aligned}\right\}.
	\end{equation}
		Notice $\sqrt{\tfrac{4c_2\Delta}{\rho+2c_2\Delta}}\leq \sqrt{\Delta}\cdot\sqrt{\frac{4c_2}{\rho}} \leq 2C' \sqrt{\Delta}$ and $(c_2\Delta+2\sqrt{c_2\Delta(\rho+2c_2\Delta)})\leq C'\sqrt{\Delta}$ for some numerical constant $C'$. For notational convenience, set $D_x:=\|x-\bar x\|+\|x+\bar x\|$.	
		 Thus, by the AM-GM inequality, the inclusion $\hat x \in B(x, C'\sqrt{\Delta}D_x)$ holds.
		
		
\begin{claim}\label{claim:upper-bound_on-size}
There exist constants $C'',\tau>0$  such that  with high probability, for all $\Delta<C''$, the inequality $\|x\|\leq \tau \|\bar x\|$ holds for any stationary point $x$ of $f_S$.
	\end{claim} 
	\begin{proof}
	We may assume that $\|\bar x\| \leq \|x\|$ since otherwise the result is trivial. Next, observe that $\|x\|$ and $\|\hat x\|$ have comparable norms: 
	\begin{align*}
\|\hat x\| &\leq \|x\| + C'\sqrt{\Delta}D_x \leq (1+4C'\sqrt{\Delta})\|x\|, \\
\|\hat x\| &\geq \|x\| - C'\sqrt{\Delta}D_x \geq (1-4C'\sqrt{\Delta})\|x\| ,
\end{align*}
	where we have used the bound $D_x \leq 4\|x\|$ twice. To make the last bound meaningful, we may set $C''<(\frac{1}{8C'})^2$, thereby ensuring $1-4C'\sqrt{\Delta}\geq 1/2$.
	Because
        the norms are comparable, we deduce
	\begin{align}
    	\dist(0;\partial f_P(\hat x))&\leq C' \sqrt{\Delta} D_x \leq
                                       4C'\sqrt{\Delta}\|x\| \leq
                                       \frac{4C'\sqrt{\Delta}}{(1-4C'\sqrt{\Delta})}\|\hat
                                       x\|. \label{eq: near_point_sub}
	\end{align}
Let us now decrease $C''$ if necessary to have $C'' < \min \{ \left ( \tfrac{1}{8C'} \right )^2, \left (
  \tfrac{\gamma}{8C'} \right )^2 \}$, where $\gamma$ is the fixed
constant from Theorem~\ref{thm:pop_obj_quant_vers_main}. Then for all
$\Delta < C''$, we have $1-4C' \sqrt{\Delta} \geq \tfrac{1}{2}$ and
$\tfrac{4C' \sqrt{\Delta}}{1-4C'\sqrt{\Delta}} \le 8C' \sqrt{\Delta}
\le \gamma$. Now we can apply
Theorem~\ref{thm:pop_obj_quant_vers_main} to $\hat x$, which guarantees that
$\|\hat x\| \lesssim \|\bar x\|$. Thus because the norms of $\| x\|$ and $\|\hat
x\|$ are comparable, we obtain the desired result. 
	\end{proof}
Provided $\Delta \le \min \{ \left ( \tfrac{1}{8C'} \right )^2, \left (
  \tfrac{\gamma}{8C'} \right )^2 \}$, we obtain from \eqref{eq:
  near_point_sub} and Claim~\ref{claim:upper-bound_on-size} the estimate
\[ \dist(0;\partial f_P(\hat x)) \le \varepsilon := C'\sqrt{\Delta}
  D_x \le 8C'
  \sqrt{\Delta} \|\hat x\| \le \gamma \|\hat x\|. \]
Applying Theorem~\ref{thm:pop_obj_quant_vers_main} we find that the point $\hat x\in B(x,C' \sqrt{\Delta}D_x)$ satisfies either
\begin{equation}
\|\hat x\|\|\hat x-\bar x\|\|\hat x+\bar x\|\lesssim \sqrt{\Delta}D_x
\|\bar x\|^2
\qquad \textrm{or}\qquad\left\{\begin{aligned}
\left|\|\hat x\|-c\|\bar x\| \right|&\lesssim
\sqrt{\Delta}D_x\frac{\|\bar x\|}{\|\hat x\|} \\
|\langle \hat x,\bar x\rangle|&\lesssim \sqrt{\Delta}D_x  \|\bar x\|
\end{aligned}\right\}. \label{eq:large_triangle}
\end{equation}
Applying the triangle inequality and the bound $D_x \leq (2+2\tau)
\|\bar x\|$, the claimed inequalities all follow (see
Appendix~\ref{sec:appendix_aux} for a detailed explanation).
%
\end{proof}

%% file: cray_computations.tex

\section{Proof of Theorem~\ref{thm:pop_obj_quant_vers_main}}\label{sec:cray_sec_inexact}
In this section, we will prove Theorem~\ref{thm:pop_obj_quant_vers_main}. Contrasting with Theorem~\ref{thm:pop_obj}, the proof of Theorem~\ref{thm:pop_obj_quant_vers_main} is much more delicate, in large part relying on perturbation bounds on eigenvalues; e.g. Gershgorin 
theorem \cite[Corollary 6.1.3]{HJ2}.
We continue using the notation of Section~\ref{sec:proof_5.2}. 
Namely,  fix a symmetric convex  function $f\colon\R^d\to\R$ and a point $\bar x\in\R^d\setminus\{0\}$, and define the function 
$$g(x):=f_\lambda(xx^T-\bar x\bar x^T).$$ 
The chain rule directly implies $$\partial g(x)=\partial f_{\lambda}(X)x.$$ 
Therefore, using Theorem~\ref{thm:char_spec_sbdiff} let us also fix a matrix $V\in\partial f_{\lambda}(X)$ and 
a matrix $U\in \mathbb{O}^d$ satisfying
$$\lambda(V)\in \partial f(\lambda(X)),\qquad V=U(\Diag (\lambda(V))U^T,\qquad \textrm{ and }\qquad X = U \Diag(\lambda(X)) U^T.$$

We begin with two technical lemmas.

\begin{lem}\label{lem:sharp_f}
	Suppose that there exists $\kappa>0$ such that the inequality 
	$$g(x)-g(\bar x)\geq \kappa \|x-\bar x\|\|x+\bar x\|\qquad \textrm{ holds for all } x\in\R^d.$$	
	Then for any $ x \notin \{\pm \bar x\}$, we have $\max\{|\lambda_1(V)|, |\lambda_d(V)|\} \geq \kappa/2$.
\end{lem}
\begin{proof}
	Using Lemma~\ref{lem:correlation}, for $i\in\{1,d\}$ we obtain
	$$
	|\lambda_i(X)| = |\dotp{U_i, x}^2 - \dotp{U_i, \bar x}^2|= |\dotp{U_i, x - \bar x} \dotp{U_i, x + \bar x}| \leq \|x - \bar x\|\|x + \bar x\|.
	$$
	Taking into account \eqref{eq:correlation_inequality}, yields
	\begin{align*}
	\kappa \|x - \bar x\| \|x + \bar x\| \leq g(x)-g(\bar x) &\leq \lambda_1(V) \lambda_1(X) + \lambda_d(V)\lambda_d(X) \\
	&\leq  2\max\{|\lambda_1(V)|, |\lambda_d(V)|\} \|x - \bar x\|\|x + \bar x\|,
	\end{align*}
	as desired.
\end{proof}

\begin{lem}\label{lem:important_rand_lem}
Suppose that there exists $\kappa>0$ such that the inequality  
\begin{equation}\label{eqn:sharpness}
	g(x)-g(\bar x)\geq \kappa \|x-\bar x\|\|x+\bar x\|\qquad \textrm{ holds for all } x\in\R^d.
	\end{equation}
Then any point $x\in \R^d\setminus\{0\}$ satisfies
\begin{align*}
\frac{\kappa\|x - \bar x\|\|x + \bar x\|}{\|x\|} - \frac{(|\lambda_{1}(V)|+|\lambda_{d}(V)|)\|\bar x\|^2}{\|x\|} \leq  \dist(0; \partial g(x)). 
\end{align*}
\end{lem}
\begin{proof}
First, note that for $x \in \{\pm \bar x\}$, the result holds trivially, so we may assume 
$x \notin \{\pm \bar x\}$. Recall the equality $\partial g(x)=\partial f_{\lambda}(X)x$. Fix now a vector $V\in \partial f_{\lambda}(X)$ satisfying $\dist(0;\partial g(x))=\|Vx\|$. Using convexity, we deduce
\begin{align}\label{eqn:conv_ipper}
 g(x)-g(\bar x)&=f_\lambda(xx^T - \bar x \bar x^T) -f_{\lambda} (0)\leq \dotp{V, xx^T - \bar x \bar x^T} \leq  \|x\| \dist(0, \partial g(x)) + |{\bar x}^TV\bar{x}|.
\end{align}
We next upper bound the term $|{\bar x}^TV\bar{x}|$. To this end, fix a matrix $U\in \mathbb{O}^d$ satisfying  $V=U\Diag(\lambda(V))U^T$ and $X=U\Diag(\lambda(X))U^T$, and such that the inclusion $\lambda(V)\in \partial f(\lambda(X))$ holds. Taking into account $\bar x\in \spann\{U_1,U_d\}$ (Lemma~\ref{lem:correlation}), we deduce 
$$|\bar x^TV\bar x|=|\lambda_{1}(V)\langle U_1,\bar x\rangle^2+\lambda_{d}(V)\langle U_d,\bar x\rangle^2|\leq (|\lambda_{1}(V)|+|\lambda_{d}(V)|)\|\bar x\|^2.$$
 Combining this estimate with \eqref{eqn:conv_ipper} and \eqref{eqn:sharpness} completes the proof.
\end{proof}


We next prove a quantitative version of Corollary~\ref{cor:find_stat}. The argument follows a similar outline.

\begin{thm}[Quantitative Version of Corollary~\ref{cor:find_stat}]\label{thm:pop_obj_quant}
	Suppose that there exists a constant $\kappa>0$ such that the inequality 
	$$g(y)-g(\bar x)\geq \kappa \|y-\bar x\|\|y+\bar x\|\qquad \textrm{ holds for all } y\in\R^d.$$
Suppose 	$|\lambda_1(V)|, |\lambda_d(V)|$ are both upper bounded by a numerical constant\footnote{This holds whenever $(t, s) \mapsto f(t, s, 0, \ldots, 0)$ is Lipschitz continuous.} and set $\varepsilon:=\|Vx\|$. Then there exists a numerical constant $\gamma > 0$, such that whenever $\varepsilon \leq \gamma \cdot \|x\|$, we have that $\|x\| \lesssim \|\bar x\|$ and $x$ satisfies either
\begin{equation*}
\|x\|\|x-\bar x\|\|x+\bar x\|\lesssim \varepsilon \|\bar x\|^2 \qquad \textrm{or}\qquad\left\{\begin{aligned}
|\lambda_1(V)|&\lesssim \varepsilon/\|x\|\\
|\langle x,\bar x\rangle|&\lesssim \varepsilon \|\bar x\|
\end{aligned}\right\}.
\end{equation*}
\end{thm}
\begin{proof}
Clearly, we may suppose $x\notin\{0,\pm\bar x\} $ and $\varepsilon \neq 0$, since otherwise the theorem would hold vacuously.
We will prove the following precise bound, which immediately implies the statement of the theorem: there exists a numerical constant $\gamma>0$, such that whenever $\varepsilon\leq \gamma\|x\|$, the inequalities $\|x\|\leq \delta\|\bar x\|$ and 
	\begin{equation}\label{eqn:crayeqn_threepart}
	\min\left\{\frac{\|x - \bar x\|\|x + \bar
          x\|}{\frac{2}{\kappa}\max\left\{\left(\frac{\|
          x\|}{\sqrt{2}}  + \frac{\sqrt{2}\|\bar x\|^2}{\|x\|}\right),
          \frac{\|x\|(\kappa \sqrt{2} + 2|\lambda_d(V)|)}{\kappa}\right\} },\max\left\{ \frac{|\lambda_1(V)| \| x\|}{\sqrt{2}}, \frac{\kappa|\dotp{x, \bar x}|}{2\sqrt{2}\delta \|x\| + 2\|\bar x\| }\right\} \right\} \leq \|V x \|,
	\end{equation}
	hold,
	where we define the numerical constant 
$$ \delta := \sqrt{\frac{2(|\lambda_{1}(V)|+|\lambda_{d}(V)|)}{\kappa}} + 1.$$ 
	
As a first step, we show that $\|x\|$ is within a numerical constant of $\|\bar x\|$. 
	\begin{claim} \label{claim:small_x_bound}
	 Provided $\gamma<\tfrac{\kappa(1-1/\delta)^2}{2}$, the inequality, $\|x\| \leq \delta \|\bar x\|$, holds.
	\end{claim}
	\begin{proof}
	Assume for sake of contradiction $\frac{\|x\|}{\|\bar
          x\|} > \delta :=
        \sqrt{\frac{2(|\lambda_{1}(V)|+|\lambda_{d}(V)|)}{\kappa}} +
        1$. Lemma~\ref{lem:sharp_f} shows $\max\{|\lambda_1(V)|,|\lambda_d(V)|\}\geq \frac{\kappa}{2}$, and therefore  $\delta >
        1$. Using the bound $\dist(0; \partial g(x)) \leq  \|Vx\| = \varepsilon$
        and Lemma~\ref{lem:important_rand_lem}, we deduce: 
	\begin{align*}
	\frac{\kappa\|x - \bar x\|\|x + \bar x\|}{\|x\|^2} -  \frac{\varepsilon}{\|x\|} \leq  \frac{(|\lambda_{1}(V)|+|\lambda_{d}(V)|)\|\bar x\|^2}{\|x\|^2}.
	\end{align*}
	Clearly, we have
	\begin{align*}
	\frac{\kappa\|x - \bar x\|\|x + \bar x\|}{\|x\|^2} \geq \frac{\kappa(\|x\| - \|\bar x\|)^2 }{\|x\|^2} \geq \kappa(1-1/\delta)^2.
	\end{align*}
	Let us now choose $\gamma<\tfrac{\kappa(1-1/\delta)^2}{2}$, thereby guaranteeing  $\tfrac{\varepsilon}{\|x\|} \leq \tfrac{\kappa(1-1/\delta)^2}{2}$. Hence, we obtain
	\begin{align*}
	\frac{\kappa(1-1/\delta)^2}{2(|\lambda_{1}(V)|+|\lambda_{d}(V)|)} &\leq \frac{1}{|\lambda_{1}(V)|+|\lambda_{d}(V)|}\left(\frac{\kappa\|x - \bar x\|\|x + \bar x\|}{\|x\|^2} -  \frac{\varepsilon}{\|x\|}\right) \leq  \frac{\|\bar x\|^2}{\|x\|^2}  < \frac{1}{\delta^2}.
	\end{align*}
	Rearranging yields 	$$ \frac{\kappa}{2(|\lambda_{1}(V)|+|\lambda_{d}(V)|)} < \frac{1} {(\delta - 1)^2},$$
	a contradiction.
	\end{proof}
	Looking back at the expression, define the values:
	$$
	\rho_1 = \frac{\|x\|}{\sqrt{2}}  \quad \text{and} \quad
        \rho_3 = \frac{2}{\kappa}\max\left\{\left(\frac{\|
              x\|}{\sqrt{2}}  + \frac{\sqrt{2} \|\bar
              x\|^2}{\|x\|}\right), \frac{\|x\|( \kappa \sqrt{2} + 2|\lambda_d(V)|)}{\kappa}\right\}.
	$$ 
Notice that the inequality, $\varepsilon \rho_3 \geq   \|x - \bar x\|\|x + \bar x\|$, would immediately imply the validity of the theorem. Thus, we assume $\varepsilon \rho_3   < \|x - \bar x\|\|x + \bar x\|$ throughout. 
	It suffices now to show 
	$$ |\lambda_1(V)| \leq \varepsilon/\rho_1\qquad \textrm{ and
        }\qquad  |\dotp{x, \bar x}|\leq
        \frac{\varepsilon}{\kappa}\left(2 \sqrt{2} \delta \|x\| + \|\bar x\|\right) .$$
	We do so in order.
	We begin by observing that the inequality~\eqref{eq:gradient_lower_bound} guarantees 
	\begin{align}\label{eq:grad_lower_bound}
	\max\{|\lambda_1(V)\dotp{U_1, x}|, |\lambda_d(V)\dotp{U_d, x}|\} \leq  \varepsilon.
	\end{align}

	\begin{claim}\label{claim1_in_weirdassproof} The inequality
		$|\lambda_1(V)| < \varepsilon/\rho_1$ holds. 
	\end{claim}
	\begin{proof}
		Let us assume the contrary,  $|\lambda_1(V)| \geq \varepsilon/\rho_1$. Inequality~\eqref{eq:grad_lower_bound}  then implies
		$|\dotp{U_1, x}| \leq \rho_1$, while Lemma~\ref{lem:correlation} in turn guarantees  
		$$ 
		0\leq \lambda_1(X) = \dotp{U_1, x}^2 - \dotp{U_1, \bar x}^2 \leq \rho_1^2.
		$$ 
		Taking into account $\dotp{U_1, x}^2 + \dotp{U_d, x}^2 = \|x\|^2$ (Lemma~\ref{lem:correlation}, correlation), we deduce $\dotp{U_d,  x}^2 \geq \|x\|^2 - \rho_1^2$. 
		Combining this with \eqref{eq:grad_lower_bound}, we deduce 
		\begin{align*}
		|\lambda_d(V)| \leq \frac{\varepsilon}{|\dotp{U_d, x}|} \leq \frac{\varepsilon}{\sqrt{\|x\|^2 - \rho_1^2}}.
		\end{align*}
		Therefore, using the correlation inequality~\eqref{eq:correlation_inequality}, we find  
		\begin{align*}
		\varepsilon \rho_3 \kappa  < \kappa\|x - \bar x\|\|x + \bar x\| &\leq g(x)-g(\bar x) \leq \lambda_1(V) \lambda_1(X) + \lambda_d(V) \lambda_d(X)  \\
		&\leq |\lambda_1(V)| (\dotp{U_1, x}^2 - \dotp{U_1, \bar x}^2) + \frac{\varepsilon}{\sqrt{\|x\|^2 - \rho_1^2}}\left(\dotp{U_d, \bar x}^2 - \dotp{U_d, x}^2\right) \\ 
		&\leq \varepsilon|\dotp{U_1, x}| + \frac{\varepsilon \dotp{U_d, \bar x}^2}{\sqrt{\|x\|^2 - \rho_1^2}} \\ 
		&\leq \varepsilon \left(\rho_1 + \frac{ \|\bar x\|^2 }{\sqrt{\|x\|^2 - \rho_1^2}}\right).  
		\end{align*}
		Dividing through by $\varepsilon$ and plugging in the value of $\rho_1$ yields
		\begin{align*}
		\rho_3 \kappa  < 
		\frac{\| x\|}{\sqrt{2}}  + \frac{\sqrt{2}\|\bar x\|^2}{\|x\|},
		\end{align*}
		which contradicts the definition of $\rho_3$. 
	\end{proof}
	Let us now decrease $\gamma>0$ further by ensuring $\gamma<\min\{\tfrac{\kappa(1-1/\delta)^2}{2},\frac{\kappa}{2 \sqrt{2}}\}$. Thus, from Claim~\ref{claim1_in_weirdassproof} and our standing assumption $\|Vx\| \leq \tfrac{\kappa\|x\|}{2\sqrt{2}}$, we conclude 
	$$
	|\lambda_1(V)|  < \frac{\sqrt{2}\varepsilon}{\|x\|} < \frac{\kappa}{2}.
	$$
	Lemma~\ref{lem:sharp_f} guarantees, $\max\{|\lambda_1(V)|, |\lambda_d(V)|\} \geq \kappa/2$; thus, we deduce $|\lambda_d(V)| \geq \kappa/2$. Applying~\eqref{eq:grad_lower_bound}, we find that 
	\begin{equation}\label{eqn:est_ud}
	|\dotp{U_d, x}| \leq \frac{\varepsilon}{|\lambda_d(V)|} \leq \frac{2\varepsilon}{\kappa}.
	\end{equation}
	Thus, by Lemma~\ref{lem:correlation}, we have 
	\begin{align}\label{eq:product_eq}
	|\dotp{U_1, \bar x}\dotp{U_d, \bar x}| = |\dotp{U_1, x}\dotp{U_d, x}| \leq \frac{2\|x\|\varepsilon}{\kappa}.
	\end{align}
	
	\begin{claim}\label{claim2_in_weirdassproof}
		The inequality $|\dotp{U_d, \bar x}| > |\dotp{U_1, \bar x}|$ holds. 
	\end{claim}
	\begin{proof}
		Let us assume the contrary $|\dotp{U_d, \bar x}| \leq |\dotp{U_1, \bar x}|$. Then from \eqref{eq:product_eq} we obtain\footnote{If $ab < \delta$, then $\min\{a, b\}^2 < \delta$.} $\dotp{U_d, \bar x}^2 < \frac{2\|x\|\varepsilon}{\kappa}$. Hence from Lemma~\ref{lem:correlation}, we find that $|\lambda_d(X) |\leq \dotp{U_d, \bar x}^2 \leq \frac{2\|x\|\varepsilon}{\kappa}$. Putting these facts together with the correlation inequality~\eqref{eq:correlation_inequality}, we successively deduce  
		\begin{align*}
		\varepsilon \rho_3 \kappa  < \kappa\|x - \bar x\|\|x + \bar x\| \leq g(x)-g(\bar x) &\leq \lambda_1(V) \lambda_1(X) + |\lambda_d(V) |\cdot |\lambda_d(X)|  \\
		&\leq\frac{\sqrt{2}\varepsilon}{\|x\|} \cdot
                  \lambda_1(X)+|\lambda_d(V) |\cdot \frac{2 \|x\|\varepsilon}{\kappa}\\
		&\leq \frac{\kappa \sqrt{2}\varepsilon\|x\|}{\kappa}+ \frac{2|\lambda_d(V)|\varepsilon\|x\|}{\kappa},
		\end{align*}
		where the last inequality uses the bound $\lambda_1(X) \leq \|x\|^2$. Therefore, we have reached a contradiction to the definition of  $\rho_3$.
	\end{proof}
	
	Combining Claim~\ref{claim2_in_weirdassproof}
	with the expression $\dotp{U_1, \bar x}^2 + \dotp{U_d, \bar
          x}^2 = \|\bar x\|^2 $, we conclude $ \dotp{U_d, \bar x}^2
        \geq \frac{\|\bar x\|^2}{2}$. Therefore,~\eqref{eq:product_eq}
        and Claim~\ref{claim:small_x_bound} imply the strong result: 
	\begin{equation}\label{eqn:upper_u1}
	|\dotp{U_1, \bar x}| \leq \frac{2\sqrt{2}\varepsilon\|x\|}{\kappa\|\bar x\|} \leq \frac{2\sqrt{2}\varepsilon \delta }{\kappa}.
	\end{equation}
	Thus combining Claim~\ref{claim:small_x_bound}, Lemma~\ref{lem:correlation}, and \eqref{eqn:est_ud}  we conclude
	\begin{align*}
	|\dotp{x, \bar x}| = |\dotp{U_1, x}\dotp{U_1, \bar x} +
                             \dotp{ U_d, x}\dotp{U_d, \bar x}| &\leq
                             |\dotp{U_1, \bar x}|\cdot\|x\| + |\dotp{
                             U_d, x}|\cdot \|\bar x\| \\
&\leq \frac{\varepsilon}{\kappa}\left(2\sqrt{2}\delta \|x\| + 2\|\bar
  x\|\right).
	\end{align*}
	The proof is complete.
\end{proof}

In order to interpret the conclusion of Theorem~\ref{thm:pop_obj_quant}  on the phase retrieval objective $f_P$, we must show that the condition 
$$\left\{\begin{aligned}
|\lambda_1(V)|&\lesssim \varepsilon/\|x\|\\
|\langle x,\bar x\rangle|&\lesssim \varepsilon \|\bar x\|
\end{aligned}\right\}$$
guarantees that the equation $\|x\|=c\cdot\|\bar x\|$ almost holds, where $c$ is defined in Theorem~\ref{thm:pop_obj}. This is the content of the following two lemmas. Note that it is easy to verify the equality $\lambda_1(V)=\nabla_{y_1} \zeta(y_1,y_2)$, where we set $(y_1,y_2):=(\lambda_1(X),\lambda_d(X))$.

	\begin{lem}[Extension of Lemma~\ref{lem:weird_comput_slope}]\label{lem:weird_comput_slope2}
		Fix a real constant $0 \le \varepsilon < 1$. The solutions of the inequality $|\nabla_{y_1}
                \zeta(y_1,y_2)| \le \varepsilon$ on $\R_{++}\times
                \R_{--}$ are precisely the elements of the open cone
                $$\{(c^2y,-y): 0<y, 0 < c_1\le c\le c_2\},$$     
where $c_1,c_2$ are the unique solutions of the equations
		$$\frac{\pi}{4} (1+ \varepsilon)
                =\frac{c_2}{1+c^2_2}+\arctan\left(c_2\right),$$
and
		$$\frac{\pi}{4} (1-\varepsilon)
                =\frac{c_1}{1+c^2_1}+\arctan\left(c_1\right).$$
Moreover, considering $c_1$ and $c_2$ as functions of $\varepsilon$, we have $c_2(\varepsilon)-c_1(\varepsilon)\leq 5\pi\varepsilon$ whenever $0 < \varepsilon < 1/2$.
	\end{lem}
	\begin{proof}
	The proof is completely analogous to that of Lemma~\ref{lem:weird_comput_slope}. We leave the details to the reader.
	The only point worth commenting is the inequality $c_2(\varepsilon)-c_1(\varepsilon)\leq 5\pi \varepsilon$ whenever $0 < \varepsilon < 1/2$. 
		 To get this bound, observe that  $0 <
                 c_2(\varepsilon) \leq c_2(.5) \leq .83$ for all
                 $\varepsilon \leq 1/2$ as $c_2$ is a increasing
                 function of $\varepsilon$. Therefore, 
\begin{equation}
\begin{aligned} \label{eq: c_bound}
\frac{\pi}{2} \varepsilon &= \frac{c_2}{1+c_2^2} - \frac{c_1}{1+c_1^2} + \arctan(c_2) - \arctan(c_1) \geq \frac{c_2}{1+c_2^2} - \frac{c_1}{1+c_1^2} = \frac{1-c_1c_2}{(1+c_1^2)(1+c_2^2)}(c_2-c_1) \\
&\geq \frac{1-c_2^2}{(1+c_1^2)(1+c_2^2)}(c_2-c_1) \geq \frac{1-c_2^2}{(1+c_2^2)^2}(c_2-c_1).
\end{aligned}
\end{equation}
Thus, we have 
$$
c_2(\varepsilon) - c_1(\varepsilon)\leq \frac{\pi\varepsilon}{2}\frac{(1+c_2^2(\varepsilon))^2}{1-c_2^2(\varepsilon)} \leq \frac{\pi\varepsilon}{2}\frac{(1+.83^2)^2}{1-.83^2} \leq 5\pi\varepsilon,
$$
as claimed.
%
%
	\end{proof}

\begin{lem}\label{lem:weird_eigenvalue_bound} Fix a real constant
  $0 \le \varepsilon < \tfrac{1}{3}$ and vectors $x, \bar{x} \in \R^d
  \setminus \{0\}$. Suppose $\lambda_1(X) = -c^2 \lambda_d(X)$ for
  some real constant $c > 0$ and $|\dotp{x, \bar{x}}| \le \varepsilon
  \|\bar x\| \|x\|$. Then we have
\begin{align*}
1- (1+c^2)(\varepsilon+\varepsilon^2)  \le c^2 \frac{\|\bar x\|^2}{\|x\|^2}
  \le 1+ (1+c^2)(\varepsilon+\varepsilon^2) .
\end{align*} 
\end{lem}

\begin{proof}
Fix a decomposition $x = \frac{\ip{x, \bar{x}}}{\|\bar x\|^2} \bar{x} + v$, where
$v \in \bar{x}^{\perp}$. Note inequality $|\dotp{x, \bar{x}}| \le \varepsilon
\|x\| \|\bar x\|$ implies that $\bar x$ and $x$ are not collinear, and therefore 
$\|v\| > 0$. Define the constant $\alpha = \frac{\ip{x,
      \bar{x}}}{\|\bar x\|^2}$. Then a quick computation shows the following decomposition:
      \begin{align*}
X = \begin{bmatrix}
\frac{\bar{x}}{\|\bar x\|} & \frac{v}{\|v\|} 
\end{bmatrix}
\begin{bmatrix}
(\alpha^2-1)\|\bar x\|^2 & \alpha \|\bar x\| \|v\| \\
\alpha \|\bar{x}\| \|v\| & \|v\|^2
\end{bmatrix} \begin{bmatrix}
\frac{\bar{x}}{\|\bar x\|} & \frac{v}{\|v\|} 
\end{bmatrix}^T.
\end{align*}
Notice that the above $2\times 2$-matrix is invertible, and therefore its eigenvalues must be $\lambda_1(X)$ and $\lambda_d(X)$.
By the Gershgorin 
theorem \cite[Corollary 6.1.3]{HJ2} applied to the $2\times 2$ matrix, we know that $\lambda_1(X)$ and $\lambda_d(X)$ must
lie in the union of the intervals
\[ \bar{D}_1 = \{z : |z-\|v\|^2| \le |\alpha| \|\bar{x}\| \|v\| \} \quad
  \text{and} \quad \bar{D}_2 = \{z \, : \, |z-(\alpha^2-1)\|\bar{x}\|^2|
  \le |\alpha| \|\bar x\| \|v\|\}.\]
 We next prove the following claim.
\begin{claim} 
The intervals $\bar{D}_1$ and $\bar{D}_2$ are contained in the following intervals around $\|x\|^2$ and $-\|\bar x\|^2$, respectively:
\begin{align*}
\bar{D}_1 \subset D_1  &:= \{ z \, : \, |z - \|x\|^2| \le 
  ( \varepsilon^2 + \varepsilon ) \|x\|^2\},\\
\bar{D}_2 \subset D_2 &:= \{z \, : \, |z+\|\bar x\|^2| \le
                        (\varepsilon^2 + \varepsilon) \|x\|^2\}.
\end{align*}
Moreover, we have $D_1 \cap D_2 = \emptyset$ and $D_1
\subset \R_{++}$.
\end{claim}
\begin{proof}
Consider the interval $\bar{D}_1$. A routine computation shows 
\begin{align*}
|\alpha| \le \frac{\varepsilon \|x\| }{\| \bar x\|}, \quad 0 \le \alpha^2 \le
  \frac{\varepsilon^2 \| x\|^2}{\| \bar x\|^2}, \quad \text{and} \quad
  0 \le \alpha \dotp{x,
  \bar{x}} \le \varepsilon^2 \|x\|^2.  
\end{align*} 
Using $\|x\| \ge \|v\|$ and $\|v\|^2 = \|x\|^2 - 2\alpha \dotp{x, \bar
  x} + \alpha^2\|\bar x\|^2$, we successively deduce for any $z \in \bar{D}_1$, the inequalities
\begin{align*}
\begin{array}{c c c c c}
-|\alpha| \|\bar{x}\| \|x\| & \le & z-\|v\|^2 &\le &|\alpha| \|\bar{x}\|
  \|x\|\\
- \varepsilon \|x\|^2 &\le& z - \|x\|^2 + 2 \alpha \dotp{x, \bar x} -
                           \alpha^2 \|\bar x\|^2 &\le &\varepsilon
                                                   \|x\|^2\\
-\varepsilon \|x\|^2 + \alpha^2 \|\bar{x}\|^2 - 2 \alpha \dotp{x, \bar
  x} & \le & z- \|x \|^2 & \le &  \varepsilon \|x\|^2  + \alpha^2 \|\bar{x}\|^2 - 2 \alpha \dotp{x, \bar
  x}\\
-\varepsilon \|x\|^2 -  \varepsilon^2 \|x\|^2 & \le & z - \|x\|^2
                                              &\le & \varepsilon
                                                     \|x\|^2 +
                                                     \varepsilon^2 \|x\|^2.
\end{array}
\end{align*}
Thus we have shown $\bar{D}_1 \subset D_1$. Similarly, for all $z \in \bar{D}_2$, we compute
\begin{align*}
\begin{array}{ccccc}
-|\alpha| \|\bar x\| \|v\| &\le& z-(\alpha^2-1) \| \bar x\|^2 &\le&
  |\alpha| \|\bar x\| \|v\|\\
-\varepsilon \|x\|^2 + \alpha^2 \|\bar x\|^2 &\le& z+ \|\bar x\|^2
                                                 &\le& \varepsilon
                                                       \|x\|^2 +
                                                       \alpha^2 \|\bar
                                                       x\|^2\\
-\varepsilon \|x\|^2 &\le& z + \|\bar x\|^2 &\le& \varepsilon \|x\|^2+ \varepsilon^2 \|x\|^2.
\end{array}
\end{align*}
We conclude $\bar{D}_2 \subset D_2$. 
Provided $\|x\| \neq 0$ and $\varepsilon^2+\varepsilon < 1$, it is clear $D_1 \subset
\R_{++}$. It remains to show that $D_2  \cap  D_1 = \emptyset$. Clearly it is sufficient to guarantee that the sum of the radii of $D_2$ and $D_1$ is strictly smaller than the distance between the centers:
$$(\varepsilon^2+\varepsilon)\|x\|^2+(\varepsilon^2+\varepsilon)\|x\|^2< \|x\|^2-(-\|\bar x\|^2).$$
Rearranging, we must guarantee $2(\varepsilon^2+\varepsilon)-1 < \frac{\|\bar x\|^2}{\|x\|^2}$. Clearly this is the case as soon as $\varepsilon<1/3$. The result follows.
\end{proof}

Thus we have proved $D_1 \cap D_2 = \emptyset$ and $D_1
\subset \R_{++}$. Since $\bar{D}_1$ and $\bar{D}_2$, each contains at least one eigenvalue, it must be the case that $\lambda_d(X)$ lies in  $\bar{D}_2$ and $\lambda_1(X)$ lies in $\bar{D}_1$.
We thus conclude
\begin{align*}
\left | \lambda_1(X)-\|x\|^2 \right | &\le (\varepsilon^2 +
                                        \varepsilon) \|x\|^2
  \\
\left | \lambda_d(X) + \|\bar x\|^2 \right | & \le
                                               (\varepsilon^2+\varepsilon) \|x\|^2.
\end{align*}
Writing $\lambda_1(X) = -c^2 \lambda_d(X)$, we obtain
\begin{align*}
\quad \left |-c^2 \lambda_d(X) - c^2 \|\bar x\|^2 + c^2 \| \bar
  x\|^2 - \|x\|^2 \right | &\le (\varepsilon^2 +
                                        \varepsilon) \|x\|^2,
\end{align*}
and hence 
$$ \left | \|x\|^2 -c^2 \|\bar x\|^2 \right | \le (1+c^2)(\varepsilon^2+\varepsilon) \|x\|^2.$$
The result follows.
\end{proof}

Combining Lemmas~\ref{lem:weird_comput_slope2} and \ref{lem:weird_eigenvalue_bound}, we arrive at the following.

\begin{cor}[Small $\lambda_1(V)$ and near orthogonality] \label{cor: nearly_orthongal_small_lambda}
Fix a real constant $0 \le \varepsilon \le \tfrac{1}{8}$ and consider a point $x\in \R^d \setminus \{0\}$ satisfying
$|\nabla_{y_1}\zeta (\lambda_1(X), \lambda_d(X))|
\le \varepsilon$ and $|\dotp{x, \bar{x}}| \le \varepsilon \|x\|
\|\bar x\|$. Then $x$ satisfies
\[\big  | \|x \|-c\| \bar x\| \big | \le 26 \varepsilon \| \bar
  x\|, \]
where $c$ is the solution of the equation 
$\frac{\pi}{4} = \frac{c}{1+c^2} + \arctan(c)$.
\end{cor}

\begin{proof} Define the quantities $c_1(\varepsilon)$ and
  $c_2(\varepsilon)$ to be the solutions of the equations
\begin{align*}
\frac{\pi}{4} (1-\varepsilon) &= \frac{c_1}{1+c_1^2} + \arctan(c_1),\\
\frac{\pi}{4} (1+\varepsilon) &= \frac{c_2}{1+c_2^2} + \arctan(c_2),
\end{align*}
respectively.
First, since $c_2(\cdot)$ is an increasing function, it is easy to verify
 $c_2(\varepsilon) < 1$ whenever 
  $0<\varepsilon \le \tfrac{1}{8}$; thus we have $2
  \varepsilon(1+c_2^2(\varepsilon)) < \frac{1}{2}$. By Lemma~\ref{lem:weird_comput_slope2}, we know that whenever $|\nabla_{y_1}
  \zeta(\lambda_1(X), \lambda_d(X))| \le \varepsilon$,  there exists $\hat{c}$
  satisfying $\lambda_1(X) = -\hat{c}^2 \lambda_d(X)$ and $0 < c_1(\varepsilon) \le
 \hat{c} \le c_2(\varepsilon)$. Lemma~\ref{lem:weird_eigenvalue_bound}, in turn, implies 
\[ (1- 2\varepsilon(1+\hat{c}^2)) \|x\|^2 \le \hat{c}^2 \|\bar{x}\|^2 \le
  (1+2\varepsilon(1+\hat{c}^2) )\|x\|^2.\]
Looking at the right-hand-side, we deduce
$$c_1^2(\varepsilon) \|\bar x\|^2 \le \hat{c}^2 \|\bar x\|^2 \le \left (
  1+2\varepsilon(1+c_2^2(\varepsilon)) \right )\|x\|^2,$$
  while looking at the left-hand-side yields
\begin{align*}
(1-2\varepsilon(1+c_2^2(\varepsilon)) ) \|x\|^2 \le \hat{c}^2 \|\bar x\|^2
                                  \le c_2^2(\varepsilon) \| \bar x\|^2.
\end{align*} 
Isolating $\|x\|^2$ and taking square roots we obtain 
\begin{equation}\label{eqn:est_nextolast}
 \frac{c_1(\varepsilon)}{\sqrt{1+2\varepsilon(1+c_2^2(\varepsilon))}}
  \|\bar x\| \le \|x\| \le \frac{c_2(\varepsilon)}{\sqrt{1-2\varepsilon(1+c_2^2(\varepsilon))}}
  \|\bar x\|.
  \end{equation}
  Applying Lemma~\ref{lem:weird_comput_slope2} and the inequality $c_2(\varepsilon)<1$, 
 we upper bound the right-hand-side:
 \begin{align*}
  \frac{c_2(\varepsilon)}{\sqrt{1-2\varepsilon(1+c_2^2(\varepsilon))}}&\leq \frac{5\pi\varepsilon +c}{\sqrt{1-4\varepsilon}}\\&= c\left(1+\frac{5\pi\varepsilon/c +1-\sqrt{1-4\varepsilon}}{\sqrt{1-4\varepsilon}}\right)\leq  c\left(1+\frac{5\pi\varepsilon/c+4\varepsilon}{\sqrt{1/2}}\right)\leq c(1+57\varepsilon).
  \end{align*}
  Exactly the same reasoning shows 
  $$ \frac{c_1(\varepsilon)}{\sqrt{1+2\varepsilon(1+c_2^2(\varepsilon))}}\geq c(1-57\varepsilon).$$
Thus the inequality $\big  | \|x \|-c\| \bar x\| \big | \le 57 c\varepsilon \| \bar
  x\|\leq 26\varepsilon \| \bar
  x\|$ holds, as claimed.
%
\end{proof}

We are now ready to prove the inexact extension of Theorem~\ref{thm:pop_obj_quant_vers_main}.

\begin{proof}[Proof of Theorem~\ref{thm:pop_obj_quant_vers_main}]
We use the decomposition $g=f_P(X)$ and $f=\varphi$. Let us verify that we may apply 
Theorem~\ref{thm:pop_obj_quant}. To this end, observe that the population objective satisfies
$$
f_P(x) - f_P(\bar x) = \EE_a\left[\left\langle a, \frac{x-\bar x}{\|x - \bar x\|}\right\rangle\left\langle a,\frac{x + \bar x}{\|x + \bar x\|}\right\rangle\right]\|x - \bar x\|\| x+ \bar x\| \geq \kappa\|x - \bar x\|\| x+ \bar x\|
$$
for the numerical constant $\kappa$~\cite[Corollary 3.7]{eM}.
Moreover, clearly $\zeta$ is globally Lipschitz (being a norm), and
therefore $|\lambda_1(V)|$ and $|\lambda_d(V)|$ are bounded by a numerical constant. Thus provided
$\varepsilon:=\|Vx\|$ satisfies $\varepsilon \leq \gamma \cdot \|x\|$ for the numerical constant $\gamma$, we can be sure that
$x$ satisfies either 
\begin{equation*}
\|x\|\|x-\bar x\|\|x+\bar x\|\lesssim \varepsilon \|\bar x\|^2 \qquad \textrm{or}\qquad\left\{\begin{aligned}
|\lambda_1(V)|&\lesssim \varepsilon/\|x\|\\
|\langle x,\bar x\rangle|&\lesssim \varepsilon \|\bar x\|
\end{aligned}\right\}.
\end{equation*}
Now suppose the latter is the case, and let $C$ be a numerical
constant satisfying $|\lambda_1(V)|\leq C \varepsilon/\|x\|$ and
$|\dotp{x, \bar x}| \le C \varepsilon \|\bar x\|$. We aim to apply Corollary~\ref{cor: nearly_orthongal_small_lambda}. To do
so, we must ensure 
\[ |\lambda_1(V)| \le \frac{C \varepsilon}{\|x\|}\le \frac{1}{8} \quad \text{and} \quad \left |
\left \langle \frac{x}{\|x\|}, \frac{\bar x}{\| \bar x\|} \right \rangle \right |
\le C \varepsilon \cdot \frac{\| \bar x\|}{\|x\| \| \bar x\|} =\frac{C\varepsilon}{\|x\|} \le \frac{1}{8}.\]
Adjusting $\gamma$ if necessary, we can be sure that
$\varepsilon/\|x\|$  is below $\frac{1}{8C}$. Applying Corollary~\ref{cor: nearly_orthongal_small_lambda}, with $\frac{C\varepsilon}{\|x\|}$ in place of $\varepsilon$, we conclude
 $|\|x\|-c\|\bar x\||\lesssim \varepsilon \frac{\|\bar
     x\|}{\|x\|} $, as claimed.
\end{proof}

\input{robust}

%% file: robust.tex
\subsection{Comments on Robustness}\label{sec:comments_robust}

We have thus far assumed that the measurement vector $b = (Ax)^2$ has not been corrupted by errant noise. In this section, we record a few straightforward extensions of earlier results, which hold if the measurements $b$ are noisy.
\begin{assumption}
Let $b_1, \ldots, b_m$ be $m$ i.i.d. copies of
\begin{align*}
\hat b = (a^T x)^2 + \delta \cdot \xi,
\end{align*}
where $\delta \in \{0, 1\}$, $\xi \in \RR$, and $a\in\R^d$ are independent random variables satisfying (1) $p_{\mathrm{fail}} := P(\delta \neq  0) < 1$, (2)  $\EE\left[|\xi|\right] < \infty$, and (3)  $a \sim \mathsf{N}(0, I_d)$. 
\end{assumption}
Under corruption by $\delta \cdot \xi$, we define new population and subsampled objectives
\begin{align*}
\hat  f_P(x) &:= \mathbb{E}_{a, \xi, \delta}[|(a^Tx)^2-(a^T\bar x)^2 - \delta \cdot \xi|];\\
\hat f_{S}(x) &:=\frac{1}{m}\sum_{i=1}^m |(a_i^Tx)^2-b_i|.
\end{align*}
Then by following the outline of the proof of Lemma~\ref{lem:spec_represe}, we arrive at a similar characterization of $\hat f_P$ as a spectral function.

\begin{lem}[Spectral representation of the population objective]\label{lem:spec_represe_robust}{\hfill \\ }
For all points $x\in\R^d$, equality holds:
\begin{equation*} 
\hat f_P(x)=\EE_{v, \xi, \delta} \left[\Big| \langle \lambda(X), v\rangle - \delta \cdot \xi_i\Big|\right],
\end{equation*}
where $v_i\in\R$ are i.i.d. chi-squared random variables $v_i\sim\chi^2_1$. 
\end{lem}

Thus, we may write
$$
\hat f_P(x) = \hat \varphi(\lambda(X)),
$$
where $\hat \varphi$ is the convex symmetric function 
$$
\hat \varphi(z) := \EE_{v, \xi, \delta} \left[\Big| \langle z, v\rangle - \delta \cdot \xi_i\Big|\right].
$$
Moreover, provided that $\bar x$ is a minimizer of $\hat f_P$, the complete set of stationary points of $\hat f_p$ may be determined from Corollary~\ref{cor:find_stat}. We prove this now.
\begin{lem}
For all $x \in \RR^d$, the following inequality holds: 
\begin{align*}
\hat f_P(x) - \hat f_P(\pm \bar x) \geq (1-2p_{\mathrm{fail}}) f_P( x).
\end{align*}
Consequently, if $p_{\mathrm{fail}} < 1/2$, the points $\pm \bar x$ are the only minimizers of $\hat f_p$, and there exists a numerical constant $\kappa$ such that
$$
\hat f_P(x) - \hat f_P(\pm \bar x) \geq \kappa(1-2p_{\mathrm{fail}})\|x - \bar x\|\|x + \bar x\|.
$$ 
\end{lem}
\begin{proof}
By expanding the difference, we find that
\begin{align*}
\hat f_P(x) - \hat f_P(\pm \bar x) &= (1-p_{\mathrm{fail}})(f_P(x) - f_P(\bar x)) + \pfail \EE_{a,  \xi}\left[ |(a^T x)^2 - (a^T\bar x)^2 - \xi | - |\xi|\right]\\
&\geq (1-p_{\mathrm{fail}})f_P(x) - \pfail\EE_{a}\left[|(a^T x)^2 - (a^T\bar x)^2|\right]\\
&\geq (1-2p_{\mathrm{fail}})f_P(x).
\end{align*}
Only the sharpness inequality is left to prove, but this is simply a consequence of the sharpness of $f_P$, which was proved in~\cite[Corollary 3.7]{eM}.
\end{proof}

Therefore, by Corollary~\ref{cor:find_stat} we arrive at the complete characterization of the stationary points of $\hat f_P$. 
\begin{thm}\label{eq:robust_stat}
The set of stationary points of $ \hat f_P$ are precisely
$$
\{\pm x\}\cup \{0\} \cup \{ x \mid \dotp{x, \bar x} = 0, \text{and } \exists \zeta \in  \partial \hat \varphi(\lambda(X)), \max_i\{\zeta_i\} = 0\}.
$$
\end{thm}
The exact location of those stationary points orthogonal to $\bar x$ depends on the structure of the convex function $\hat \varphi$, which in turn depends the distribution of the noise $\delta \cdot \xi$. We will not attempt to characterize such $\hat \varphi$.

By the sharpness of $\hat f_P$, a quantitative version of Theorem~\ref{eq:robust_stat} immediately follows from Theorem~\ref{thm:pop_obj_quant}. When coupled together with a concentration inequality like that in Theorem~\ref{thm:em}, such a theorem would imply concentration of the subdifferential graphs of $\hat f_S$ and $\hat f_P$. We omit these straightforward details.